\documentclass{amsart}

\usepackage{geometry}                % See geometry.pdf to learn the layout options. There are lots.
\usepackage{graphicx}
\usepackage{amssymb}
\usepackage{epstopdf}
\usepackage{amsmath}
\usepackage{amsthm}
\usepackage{caption}
\usepackage{subcaption}
\usepackage{graphicx}
\usepackage{listofsymbols}
\usepackage{fix-cm}
\usepackage{amsmath,amsfonts,mathrsfs}
\usepackage{graphicx}
\usepackage{graphicx}
\usepackage{amssymb}
\usepackage{hyperref}
\usepackage{epstopdf}
\usepackage{pgf}
\usepackage{url}
\usepackage{tikz}
\usepackage{latexsym}
\usepackage{geometry}                % See geometry.pdf to learn the layout options. There are lots.                  % ... or a4paper or a5paper or ...
\usepackage{graphicx}
\usepackage{amssymb}
\usepackage{epstopdf}
\usepackage{pgf}
\usepackage{latexsym}
\usepackage{amsmath,amsfonts}
\usepackage{amssymb}
\usepackage{epstopdf}
\usepackage{pgf,tikz}
\usepackage{latexsym}
\usepackage{color}
\usepackage{geometry}                % See geometry.pdf to learn the layout options. There are lots.
\usepackage{graphicx}
\usepackage{amssymb}
\usepackage{epstopdf}
\usepackage{pgf,tikz}
\usepackage{amsmath}
\usepackage{tikz}
\usepackage{graphicx}
\usepackage{graphicx}
\usepackage{epstopdf}
\usepackage{fix-cm}
\usepackage{geometry}                % See geometry.pdf to learn the layout options. There are lots.
\usepackage{graphicx}
\usepackage{amssymb}
\usepackage{epstopdf}
\usepackage{pgf,tikz}
\usepackage{latexsym}
\usepackage{amsmath,amsfonts}
\usepackage{graphicx}
\usepackage{geometry}
\usepackage{graphicx}
\usepackage{multirow}
\usepackage{amssymb}
\usepackage{epstopdf}
\usepackage{pgf,tikz}
\usepackage{latexsym}
\usepackage[english]{babel}
\usepackage{dcolumn}
\usepackage{color}
\usepackage{fancyhdr}
\usepackage{longtable}
\usepackage{array}

\begin{document}

\numberwithin{equation}{section}
\def\la{\langle}
\def\ra{\rangle}
\def\glexe{\leq_{gl}\,}
\def\glex{<_{gl}\,}
\def\e{{\rm e}}

\def\fac{{\rm !}}
\def\x{\mathbf{x}}
\def\P{\mathbb{P}}
\def\S{\mathbf{S}}
\def\h{\mathbf{h}}
\def\y{\mathbf{y}}
\def\bz{\mathbf{z}}
\def\F{\mathcal{F}}
\def\R{\mathbb{R}}
\def\T{\mathbf{T}}
\def\N{\mathbb{N}}
\def\D{\mathbf{D}}
\def\V{\mathbf{V}}
\def\U{\mathbf{U}}
\def\K{{\mathcal{K}}}
\def\Q{\mathbf{Q}}
\def\M{\mathbf{M}}
\def\oM{\overline{\mathbf{M}}}
\def\O{\mathbf{O}}
\def\C{\mathbb{C}}
\def\P{\mathbb{P}}
\def\Z{\mathbb{Z}}
\def\H{\mathcal{H}}
\def\A{\mathbf{A}}
\def\V{\mathbf{V}}
\def\AA{\overline{\mathbf{A}}}
\def\B{\mathbf{B}}
\def\c{\mathbf{C}}
\def\L{\mathcal{L}}
\def\bS{\mathbf{S}}
\def\H{\mathcal{H}}
\def\I{\mathbf{I}}
\def\Y{\mathbf{Y}}
\def\X{\mathbf{X}}
\def\G{\mathbf{G}}
\def\f{\mathbf{f}}
\def\z{\mathbf{z}}
\def\v{\mathbf{v}}
\def\y{\mathbf{y}}
\def\d{\hat{d}}
\def\x{\mathbf{x}}
\def\bI{\mathbf{I}}
\def\y{\mathbf{y}}
\def\g{\mathbf{g}}
\def\w{\mathbf{w}}
\def\b{\mathbf{b}}
\def\a{\mathbf{a}}
\def\1{\mathbf{1}}
\def\u{\mathbf{u}}
\def\q{\mathbf{q}}
\def\e{\mathbf{e}}
\def\s{\mathcal{S}}
\def\cc{\mathcal{C}}
\def\co{{\rm co}\,}
\def\tg{\tilde{g}}
\def\tx{\tilde{\x}}
\def\tg{\tilde{g}}
\def\tA{\tilde{\A}}

\def\supmu{{\rm supp}\,\mu}
\def\supp{{\rm supp}\,}
\def\cd{\mathcal{C}_d}
\def\cok{\mathcal{C}_{\K}}
\def\cop{COP}
\def\vol{{\rm vol}\,}

\newcommand{\ignore}[1]{}

\newcommand{\HH}{{\mathcal H}}

\newtheorem{theorem}{Theorem}[section]
\newtheorem{claim}[theorem]{Claim}
\newtheorem{assumption}{Assumption}
\newtheorem{definition}[theorem]{Definition}
\newtheorem{corollary}[theorem]{Corollary}
\newtheorem{property}[theorem]{Property}
\newtheorem{proposition}[theorem]{Proposition}
\newtheorem{lemma}[theorem]{Lemma}
\newtheorem{remark}[theorem]{Remark}
\newtheorem{example}[theorem]{Example}
\newtheorem{algorithm}{Algorithm}
\newtheorem{question}{Question}
\newtheorem{conjecture}{Conjecture}

\newcommand{\pmax}{p_{\text{\rm max}}}
\newcommand{\phan}{p_{\text{\rm han}}}
\newcommand{\tH}{\tilde H}
\newcommand{\ti}{\tilde}
\newcommand{\oR}{{\mathbb R}}
\newcommand{\oN}{{\mathbb N}}
\newcommand{\oP}{{\mathbb P}}
\newcommand{\cF}{{\mathcal F}}
\newcommand{\pmin}{p_{\min}}
\newcommand{\PP}{{\mathcal P}}
\newcommand{\II}{{\mathcal I}}
\newcommand{\ux}{\underline{x}}
\newcommand{\plas}{p_{\text{\rm las}}}
\newcommand{\BB} {{\mathcal B}}
\newcommand{\CC}{{\mathcal C}}
\newcommand{\tcolred}{\textcolor{black}}
\newcommand{\tcolgreen}{\textcolor{green}}
\newcommand{\tcolblue}{\textcolor{blue}}
\newcommand{\tcolmag}{\textcolor{magenta}}
\newcommand{\sfT}{\sf T}
\newcommand{\sa}{\text{\rm sa}}

%\title[Global optimization of polynomials over simple stes]{Global optimization of polynomials over simple sets via elementary calculations}

\title{Bound-constrained polynomial optimization using only elementary calculations}
\author{Etienne de Klerk
 \and Jean B. Lasserre
\and Monique Laurent
\and Zhao Sun
}

\address{Tilburg University and Delft University of Technology;
PO Box 90153, 5000 LE Tilburg, The Netherlands.}
\email{E.deKlerk@uvt.nl}
\address{LAAS-CNRS and Institute of Mathematics,
University of Toulouse,
LAAS, 7 avenue du Colonel Roche,
31077 Toulouse C\'edex 4, France,
Tel: +33561336415}
\email{lasserre@laas.fr}
\address{Centrum Wiskunde \& Informatica (CWI), Amsterdam and Tilburg University;
CWI, Postbus 94079, 1090 GB Amsterdam, The Netherlands.}
\email{M.Laurent@cwi.nl}
\address{\'Ecole Polytechnique de Montr\'eal;
Canada Excellence Research Chair in Data Science for Real-Time Decision-Making, C.P. 6079, Succ. Centre-ville, Montr¨¦al, H3C 3A7, Canada.}
\email{Zhao.Sun@polymtl.ca}

\date{}
\keywords{Polynomial optimization \and bound-constrained optimization\and  Lasserre hierarchy}
\subjclass[2000]{90C22\and 90C26 \and 90C30}
\begin{abstract}
We provide a monotone non-increasing
sequence of upper bounds $f^H_k$ ($k\ge 1$) converging to the global minimum of a polynomial $f$ on simple sets like
the unit hypercube in $\mathbb{R}^n$.
The novelty with respect to the converging sequence of upper bounds in
[J.B.  Lasserre,  A new look at nonnegativity on closed sets and polynomial optimization,  \emph{SIAM J. Optim.} {\bf 21},  pp. 864--885, 2010]
is that only elementary computations are required.
For optimization over the hypercube  $[0,1]^n$, we show that the new bounds $f^H_k$ have a rate of convergence in $O(1/\sqrt {k})$. Moreover we show a stronger convergence rate  in $O(1/k)$ for quadratic polynomials and more generally for polynomials having a rational minimizer in the hypercube.
{In comparison,  evaluation of all rational grid points with denominator $k$ produces bounds with
 a rate of convergence in $O(1/k^2)$,  but at the cost of $O(k^n)$ function evaluations, while the new bound $f^H_k$ needs only $O(n^k)$ elementary calculations.}
%We also provide error bounds for the sequence,
%which show that the new method compares favourably to evaluating the objective function on a fixed grid.

% \PACS{PACS code1 \and PACS code2 }

\end{abstract}

\maketitle

\section{Introduction}

Consider the problem of computing the global minimum
\begin{equation}\label{pb-def}
f_{\min,\K} \:=\,\min\,\{f(\x):\: \x\in\K\,\},
\end{equation}
of a polynomial  $f$ on a compact set $\K\subset \mathbb{R}^n$. (We will mainly deal with the case where $\K$ is a basic semi-algebraic set.)

A fruitful perspective, introduced by Lasserre \cite{lass-siopt-01}, is to reformulate problem \eqref{pb-def} as
\[
f_{\min,\K} = \inf_{\mathbf{\mu}} \int_{\K} fd\mathbf{\mu},
\]
where the infimum is taken over all probability measures $\mu$ with support in $\K$.
Using this reformulation one may obtain a sequence of \emph{lower bounds} on $f_{\min,\K}$ that converges to $f_{\min,\K}$,
by introducing tractable convex relaxations of the set of probability measures with support in $\K$ (if $\K$ is semi-algebraic).
For more details on this approach the interested reader is referred to
Lasserre \cite{cras,lass-siopt-01,lass-book}, and   \cite{laurent,Lasserre_MOR_2002} for a comparison between linear programming (LP)  and
semidefinite programming
(SDP) relaxations.

As an alternative, one may obtain a sequence of {\em upper bounds} by optimizing over specific classes of probability distributions.
In particular, Lasserre \cite{lass-siopt-10} defined the sequence (also called hierarchy) of upper bounds
\begin{equation}\label{eqfsosk}
f_k^{sos} := \min_{\sigma \in \Sigma_k[\x]} \left\{ \int_{\K} f(\x)\sigma(\x)d\x \; : \; \int_{\K} \sigma(\x)d\x = 1\right\}, \quad (k= 1,2,\ldots),
\end{equation}
where $\Sigma_k[\x]$ denotes the cone of sums of squares (SOS) of polynomials of degree at most $2k$. Thus the optimization is restricted to
probability distributions where the probability density function is an SOS polynomial of degree at most $2k$.
Lasserre \cite{lass-siopt-10} showed that $f_k^{sos} \rightarrow f_{\min,\K}$ as $k \rightarrow \infty$ (see Theorem \ref{th-lasserre} below for a precise statement).
 In principle this approach works for any
 compact set $\K$ and any polynomial but for practical implementation it requires knowledge of moments of the measure $\sigma(\x)d\x$.
  So in practice the approach is limited to {\it simple} sets $\K$ like the Euclidean ball, the hypersphere, the simplex, the hypercube and/or their
image by a linear transformation.

In fact computing such upper bounds reduces to computing the smallest generalized eigenvalue associated with two
 real symmetric matrices whose size increases in the hierarchy.
For more details the interested reader is referred to Lasserre \cite{lass-siopt-10}. In a recent paper,
De Klerk et al. \cite{deklerk} have provided the first convergence analysis for this hierarchy and shown a bound $f_k^{sos} - f_{\min,\K} = O(1/\sqrt{k})$ on the rate of convergence.
In a related analysis of convergence
Romero and Velasco \cite{romero} provide a bound on the rate at which one may approximate from outside
the cone of nonnegative homogeneous polynomials (of fixed degree)  by the hierarchy of spectrahedra defined in
\cite{lass-siopt-10}.
%\footnote{\tcolred{@Jean: Can we say a bit more specific fact about their result? }}

{It should be emphasized that it is a difficult challenge in optimization to obtain a sequence of upper bounds converging to the global minimum
and having a known estimate on the rate of convergence.}
%It should be emphasized that obtaining a sequence of upper bounds converging to the global minimum
%is a difficult challenge in optimization (usually upper bounds are obtained via iterates of a ``descent" algorithm of constrained
% optimization but \tcolred{with no known estimate on the rate of convergence}).
 %\footnote{E.g.g. the algorithm in Baron has a proof of convergence!}
 %guarantee of convergence to the  global minimum).
So even if the convergence to the global minimum of the hierarchy of upper bounds obtained in \cite{lass-siopt-10}
is rather slow, and even though it applies to the restricted context of ``simple sets",
to the best of our knowledge it provides { one of the first results} of this kind.
{A notable earlier result
was obtained for polynomial optimization over the simplex, where it has been shown that brute force grid search leads to a polynomial time approximation scheme for minimizing polynomials of fixed degree \cite{BK02,KLP06}. When minimizing over the set of grid points in the standard simplex with given denominator $k$, the rate of convergence is in $O(1/k)$ \cite{BK02,KLP06} and, for quadratic polynomials (and for general polynomials having  a rational minimizer), in $O(1/k^2)$ \cite{KLS14}.  Grid search over the hypercube was also shown to have a   rate of convergence in $O(1/k)$ \cite{KL10} and, as we will indicate in this paper, a stronger rate of convergence in $O(1/k^2)$ can be shown. Note however that computing the best grid point in the hypercube $[0,1]^n$ with denominator $k$ requires $O(k^n)$ computations, thus exponential in the dimension.}

\subsection*{Contribution}
{As our main contribution} we provide a monotone non-increasing converging sequence $(f^H_k)_{k\in\N}$,
of upper bounds $f^H_k\geq f_{\min,\K}$ such that $f^H_k\to f_{\min,\K}$ as $k\to\infty$.
{The parameters $f^H_k$ can be effectively computed}
when the set $\K \subseteq [0,1]^n$ is a ``simple set'' like,  for example, a Euclidean ball,  sphere,  simplex,  hypercube,
or any linear transformation of them.

This ``hierarchy" of upper bounds is inspired from
the one defined by Lasserre in \cite{lass-siopt-10}, but with the novelty that:
\begin{center}
{\it Computing the upper bounds
$(f^H_k)$ does not require solving an SDP or computing the smallest generalized eigenvalue of some pair of matrices (as is the case in \cite{lass-siopt-10}).
It only requires elementary calculations (but possibly many of them for good quality bounds).}
\end{center}
Indeed, computing the upper bound $f^H_k$ only requires finding the minimum in a list of
$O(n^k)$ scalars $(\gamma_{(\eta,\beta)})$, formed from the moments $\boldsymbol{\gamma}$ of the Lebesgue measure on the set $\K\subseteq [0,1]^n$ and from
the coefficients $(f_\alpha)$ of the  polynomial $f$ to minimize. Namely:
\begin{equation}
\label{maxcutformula}
f^H_k\,:=\,\displaystyle\min_{(\eta,\beta)\in\N^{2n}_k}\:\sum_{\alpha\in\N^{n}}f_{\alpha}\,\frac{\gamma_{(\eta+\alpha,\beta)}}{\gamma_{(\eta,\beta)}},\end{equation}
where $\N$ denotes the nonnegative integers, $f(\x)=\sum_{\alpha\in\N^n}f_\alpha\,\x^\alpha$, $\N^{2n}_k=\{(\eta,\beta)\in\N^{2n}:\vert\eta+\beta\vert=k\}$, and the scalars
\[\gamma_{(\eta,\beta)}\,:=\,\int_\K x_1^{\eta_1}\cdots x_n^{\eta_n}
(1-x_1)^{\beta_1}\cdots (1-x_n)^{\beta_n}\,d\x,\quad (\eta,\beta)\in\N^{2n},\]
are available in closed-form. (Our informal notion of ``simple set'' therefore means that the moments $\gamma_{(\eta,\beta)}$ are known a priori.)

The upper bound (\ref{maxcutformula}) has also a simple interpretation as it reads:
\begin{equation}\label{eqfHk}
f^H_k=\,\displaystyle\min_{(\eta,\beta)\in\N^{2n}_k}\:\frac{\displaystyle\int_\K f(\x)\,\x^\eta(\1-\x)^\beta\,d\x}
{\displaystyle\int_\K \x^\eta(\1-\x)^\beta\,d\x}
\,=\,\displaystyle\min_{\mu}\,\left\{\displaystyle\,\int_\K f\,d\mu:\:\mu\in M(\K)_k\,\right\},
\end{equation}
where $M(\K)_k$ is the set of probability measures on $\K$, absolutely continuous with respect to the Lebesgue measure on $\K$,
and whose density is a monomial $\x^\eta(\1-\x)^\beta$ with $(\eta,\beta)\in\N^{2n}_k$.
(Such measures are in fact products of (univariate) beta distributions, see Section \ref{secbeta}.)
This also proves that
at any point $\a\in [0,1]^n$ one may approximate the Dirac measure $\delta_\a$
with measures of the form  $d\mu=\x^\eta(\1-\x)^\beta\,d\x$ (normalized to make then probability measures).

{For the case of the hypercube $\K=[0,1]^n$, we analyze  the rate of convergence of the bounds $f^H_k$ and
 show a rate of convergence in $O(1/\sqrt {k})$ for general polynomials, and in $O(1/k)$ for quadratic polynomials (and general polynomials having a rational minimizer).
As a second minor contribution, we revisit grid search over the rational points with given denominator $k$ in the hypercube and observe that its convergence rate is in $O(1/k^2)$ (which follows as an easy application of Taylor's theorem).  However as observed earlier the computation of the best grid point with denominator $k$ requires $O(k^n)$ function evaluations  while the computation of the parameter $f^H_k$ requires only $O(n^k)$ elementary calculations.}

\subsection*{Organization of the paper.}
%The main contents of our paper are as follows.
 We start with  some basic facts about the bounds $f^H_k$ in Section \ref{secprel} and in Section \ref{secconv} we show their convergence to the minimum of $f$ over the set $\K$ (see Theorem \ref{th-main}).

%In the rest of the paper we mainly consider the case when $K=[0,1]^n$.
In Section \ref{secerror}, for the case of the hypercube $\K=[0,1]^n$, we analyze the quality of the bounds $f^H_k$.
We show a convergence rate in $ O(1/\sqrt{k})$ for the range $f^H_k - f_{\min,\K}$ and a stronger convergence rate
 in $O(1/k)$ when  the polynomial $f$ admits a rational minimizer in $[0,1]^n$  (see Theorem \ref{cor:convergence rate}).
  This stronger  convergence rate applies in particular to  quadratic polynomials (since they have a rational minimizer) and  {Example \ref{extight} shows that this bound is tight. When no rational minimizer exists the weaker rate follows using Diophantine approximations.}
%The main ingredient is (roughly) to construct a suitable pair $(\eta,\beta)\in \N^{2n}_k$ from a given global minimizer $x^*$ of $f$ over $K$.
So again the main message of this paper is that one may obtain non-trivial upper bounds with error guarantees (and converging to the global minimum) via elementary calculations
and without invoking a sophisticated algorithm.

{In Section \ref{secgrid} we revisit the simple technique which consists of evaluating the polynomial $f$ at all  rational points in $[0,1]^n$ with given denominator $k$. By a simple application of Taylor's theorem we can show a  convergence rate  in $O(1/k^2)$.
%This technique is superior to brute force evaluation on a grid over the box $[0,1]^n$
However, in terms of computational complexity, the parameters $f^H_k$ are easier to compute. Indeed, for fixed $k$, computing $f^H_k$ requires $O(n^k)$ computations (similar to function evaluations),
while computing the minimum of $f$ over all grid points with given denominator $k$ requires an exponential number $k^n$ of function evaluations. }
%On the other hand, both bounds have the same convergence rate in $O(1/k)$ (when a rational minimizer exists) (see Section~\ref{secgrid}).
%This is confirmed on a sample of (limited) experiments on $K=[0,1]^n"$, which we present in Section \ref{sec:numerical results}.

In Section \ref{sec:feasible points} we present some  additional (simple) techniques to provide a feasible point $\hat \x \in \K$ with value $ f(\hat \x) \le  f^H_k$,
once the  upper bound $f^H_k$ has been computed,  hence also with an error bound guarantee in the case of the box $\K=[0,1]^n$.
 This includes, in the case when $f$ is convex, getting a feasible point using Jensen inequality (Section \ref{secconvex}) and, in the general case, taking the mode $\hat \x$ of the optimal  density   function (i.e., its global maximizer) (see Section \ref{secmode}).

 In Section \ref{sec:numerical results}, we present some numerical experiments, carried out on several test functions on the box $[0,1]^n$.
In particular, we compare the values of the new bound $f^H_k$
%(which can be interpreted as using a Handelman-type density function, see Lemma \ref{lemma:f^H reformulation})
 with the  bound $f^{sos}_{k/2}$ (whose definition uses a sum of squares density), and we apply the proposed techniques to find a feasible point in the box. As expected the sos based bound is tighter in most cases but the bound $f^H_k$ can be computed for much larger values of $k$. Moreover, the feasible points $\hat \x$ returned by the proposed mode heuristic are often of very good quality for sufficiently large $k$.
 %\footnote{\tcolred{Check this last sentence.}}
Finally, in Section \ref{sec:conclusion} we conclude with some remarks on variants of the bound $f^H_k$ that may offer better results in practice.

%We indicate in Section \ref{secpoint} several procedures that might be used for constructing a feasible point $x\in K\subseteq [0,1]^n$ whose value $f(x)$ is at least as good as the bound $f^H_k$.

\section{Notation, definitions and preliminary results}\label{secprel}

Throughout we let $\R[\x]$ denote the ring of polynomials in the variables $\x=(x_1,\ldots,x_n)$,
$\R[\x]_d$ is the subspace of polynomials of degree at most $d$, and
$\Sigma[\x]_d\subset\R[\x]_{2d}$ is its subset of sums of squares (SOS) of degree at most $2d$.

We use the convention that $\N$ denotes the set of nonnegative integers, and
set $\N^n_d:=\{\alpha\in \N^n:\sum_{i=1}^n\alpha_i\,(=:\vert\alpha\vert)= d\}$, and similarly
$\N^n_{\le d}:=\{\alpha\in \N^n:\sum_{i=1}^n\alpha_i\,\le d\}$.
The notation $\x^\alpha$ stands for the monomial $x_1^{\alpha_1}\cdots x_n^{\alpha_n}$,
while $(\1-\x)^\alpha$ stands for $(1-x_1)^{\alpha_1}\cdots (1-x_n)^{\alpha_n}$, $\alpha\in\N^n$.
We will also denote $[n] = \{1,2,\ldots,n\}$ and let $\mathbf 1$ denote the all-ones vector (of suitable size).

One may write every polynomial $f\in\R[\x]_d$ in the monomial basis
\[\x\mapsto f(\x)\,=\,\sum_{\alpha\in\N^{n}_{\le d}}f_{\alpha}\,\x^\alpha,\]
with vector of (finitely many) coefficients $(f_\alpha)$.

\subsection{The bounds $f^{sos}_k$ and $f^H_k$}
In \cite{lass-siopt-10}, Lasserre introduced the parameters $f^{sos}_k$ as upper bounds for the minimum $f_{\min,\K}$ of $f$ over $\K$ and he proved the following result.
\begin{theorem}[Lasserre \cite{lass-siopt-10}]
\label{th-lasserre}
Let $\K\subseteq \R^n$ be compact, let $f_{\min,\K}$ be as in (\ref{pb-def}), and let
\begin{equation}
\label{th-eq1}
f^{sos}_k\,:=\,\inf_{\sigma}\,\left\{\displaystyle\int_\K f(\x)\,\sigma(\x)\,d\x:\:
\displaystyle\int_\K \sigma(\x)\,d\x=1,\sigma\in\Sigma[\x]_k\,\right\},\quad k\in \N.
\end{equation}
Then $f_{\min,\K}\,\leq\,f^{sos}_{k+1}\,\leq\, f^{sos}_{k}$ for all $k$ and
\begin{equation}
\label{th-eq2}
f_{\min,\K}\,=\,\displaystyle\lim_{k\to\infty}\:f^{sos}_k.
\end{equation}
\end{theorem}

We will also use the following important result due to Krivine \cite{krivine1,krivine2} and Handelman \cite{handelman}.
\begin{theorem}
\label{thm:Handelman}
Let $\K=\{\x:g_j(\x)\geq0,\:j=1,\dots,m\}\subset\R^n$ be a polytope with a nonempty interior and
where each $g_j$ is an affine polynomial, $j=1,\ldots,m$. If
$f\in\R[\x]$ is strictly positive on $\K$ then
\begin{equation}\label{krivine}
f(\x)\,=\,\sum_{\alpha\in\N^m}\lambda_\alpha\,g_1(\x)^{\alpha_1}\cdots g_m(\x)^{\alpha_m},\qquad\forall \x\in\R^n,
\end{equation}
for finitely many positive scalars $\lambda_\alpha$.
\end{theorem}
We will call the expression in \eqref{krivine} the {\em Handelman representation} of $f$, and call any $f$ that allows a Handelman representation to be {\em of the Handelman type}.
Throughout we consider the following set of polynomials: % of the form:
\begin{equation}\label{eqHQ}
\HH_k := \left\{p \in \R[\x]: % \; \left| \;
p(\x) = \sum_{(\eta,\beta)\in\N^{2n}_k} \lambda_{\eta,\beta} \x^\eta(\1-\x)^\beta \ \ \text{ where } \lambda_{\eta\beta}\ge 0\right\}, %\right.
\end{equation}
i.e., all polynomials that admit a Handelman representation of degree at most $k$ in terms of the polynomials $x_i, 1-x_i$ defining the hypercube $[0,1]^n$.

 Observe that any term  $\x^\eta(\1-\x)^\beta$ with degree $|\eta+\beta| <k$ also belongs to the set $\HH_k$.
 This follows  by iteratively applying  the identity: $1=x_i + (1-x_i)$, which permits
to rewrite  $\x^\eta(\1-\x)^\beta$  as a conic combination of terms $\x^{\eta'}(\1-\x)^{\beta'}$ with degree $|\eta'+\beta'|=k$.
 The next claim follows then as  a direct application.

\begin{lemma}  \label{setHk}
  We have the inclusion: $\HH_k\subseteq \HH_{k+1}$ for all $k$.
\end{lemma}

We may now interpret the new upper bounds $f^H_k$ from \eqref{maxcutformula}  in an analogous way as the bounds $f^{sos}_k$ from  \eqref{th-eq1}, but where the SOS density function $\sigma \in \Sigma_k[\x]$ is now replaced
by a density $\sigma \in \HH_k$.

{For clarity we first repeat the definition \eqref{maxcutformula} of the  parameters $f^H_k$ below:
\begin{equation*}
%\label{maxcutformula}
f^H_k\,:=\,\displaystyle\min_{(\eta,\beta)\in\N^{2n}_k}\:\sum_{\alpha\in\N^{n}}f_{\alpha}\,\frac{\gamma_{(\eta+\alpha,\beta)}}{\gamma_{(\eta,\beta)}},
\end{equation*}
where  the scalars
\[\gamma_{(\eta,\beta)}\,=\, \int_{\mathcal K} \x^\eta(\mathbf 1-\x)^\beta d\x = \,\int_\K x_1^{\eta_1}\cdots x_n^{\eta_n}
(1-x_1)^{\beta_1}\cdots (1-x_n)^{\beta_n}\,dx_1 \cdots dx_n,\quad (\eta,\beta)\in\N^{2n},\]
denote the moments of the Lebesgue measure on the set $\mathcal K$.
Using the fact that
$$\sum_{\alpha\in\N^{n}}f_{\alpha} \gamma_{(\eta+\alpha,\beta)}
= \sum_{\alpha\in\N^{n}}f_{\alpha} \int_{\mathcal K} \x^{\eta+\alpha} (\mathbf 1 -\x)^\beta d\x
=\int_{\mathcal K} f(\x) \x^\eta (\mathbf 1 -\x)^\beta d\x,$$
we can rewrite the parameter $f^H_k$ as in \eqref{eqfHk}:
\begin{equation*}
%\label{eqfHk}
f^H_k=\,\displaystyle\min_{(\eta,\beta)\in\N^{2n}_k}\:\frac{\displaystyle\int_\K f(\x)\,\x^\eta(\1-\x)^\beta\,d\x}
{\displaystyle\int_\K \x^\eta(\1-\x)^\beta\,d\x}.
%\,=\,\displaystyle\min_{\mu}\,\left\{\displaystyle\,\int_\K f\,d\mu:\:\mu\in M(\K)_k\,\right\},
\end{equation*}
We now give yet another reformulation for the parameter $f^H_k$, where we optimize over density functions in the set $\mathcal H_k$, which turn out to be convex combinations of density functions of the form
$\x^\eta(\mathbf 1-\x)^\beta$ (after suitable scaling).}

\begin{lemma}
\label{lemma:f^H reformulation}
Let $\K\subseteq [0,1]^n$, let $f$ be a polynomial, and
consider the parameters $f^H_k$, $k\in\N$, from \eqref{maxcutformula}. Then one has:
\[
f^H_k = \,\inf_{\sigma\in\HH_k\,}\,\left\{\displaystyle\int_\K f(\x)\,\sigma(\x)\,d\x:\:
\displaystyle\int_\K \sigma(\x)\,d\x=1\right\}\quad \text{ for all } k\in \N,
\]
and the sequence $(f^H_k)_k$ is monotonically non-increasing: $f^H_{k+1}\le f^H_k$.
\end{lemma}
\begin{proof}
Note that, for given $k \in \N$,
\begin{eqnarray*}
&&\inf_{\sigma}\,\left\{\displaystyle\int_\K f(\x)\,\sigma(\x)\,d\x:\:
\displaystyle\int_\K \sigma(\x)\,d\x=1, \; \sigma\in\HH_k\,\right\} \\
&=& \inf_{\lambda\geq0}\left\{\sum_{\alpha\in\N^n_d} f_\alpha\left(
\sum_{(\eta,\beta)\in\N^{2n}_k}\lambda_{\eta\beta}\,
\underbrace{\int_\K \x^{\eta+\alpha}(\1-\x)^{\beta}\,d\x}_{\gamma_{(\eta+\alpha,\beta)}}\right):
\sum_{(\eta,\beta)\in\N^{2n}_k}\lambda_{\eta\beta}
\underbrace{\int_\K \x^{\eta}(\1-\x)^{\beta}\,d\x}_{\gamma_{(\eta,\beta)}} =1\right\} \\
&=&\displaystyle\inf_{\lambda\geq 0}
\left\{
\displaystyle\sum_{(\eta,\beta)\in\N^{2n}_k}\lambda_{\eta\beta}\,\left(\displaystyle\sum_{\alpha\in\N^n_d}f_\alpha\,\gamma_{(\eta+\alpha,\beta)}\right):\quad
\displaystyle\sum_{(\eta,\beta)\in\N^{2n}_k}\lambda_{\eta\beta}\,\gamma_{(\eta,\beta)}=1\right\}\\
&=&\displaystyle\min_{(\eta,\beta)\in\N^{2n}_k}\:\sum_{\alpha\in\N^{n}_d}f_{\alpha}\,\frac{\gamma_{(\eta+\alpha,\beta)}}{\gamma_{(\eta,\beta)}} = f^H_k,
\end{eqnarray*}
where we have used the fact that the penultimate optimization problem is an LP over a simplex that attains its infimum at one of the vertices.
The monotonicity of the sequence $(f^H_k)_{k\in\N}$ now follows from Lemma \ref{setHk}.
\end{proof}

\subsection{Calculating moments on $\K$}
For $\K\subseteq [0,1]^n$  a compact set and for every $(\eta,\beta)\in\N^{2n}$, we need to calculate the parameters
\begin{equation}\label{gene-moments}
\gamma_{(\eta,\beta)}\,:=\,\int_\K \x^\eta (\1-\x)^\beta\,d\x,\end{equation}
in order to compute $f^H_k$.
When $\K$ is arbitrary one does not know how to compute such
generalized moments. But if $\K$ is the unit hypercube $[0,1]^n$, the simplex
$\Delta:=\{\x: \x\geq0;\,\sum_{i=1}^nx_i\leq1\}$,
a Euclidean ball (or sphere),
%the hypercube $\{0,1\}^n$
and/or their image by a linear mapping,
then such moments are available in closed-form; see e.g.\ \cite{lass-siopt-10}.
{We give the moments for the hypercube $\K=[0,1]^n$, which we will treat in detail in this paper. Namely,}
%For instance, for the hypercube $\K=[0,1]^n$, we have
\[\int_{[0,1]^n} \x^\eta\,(\1-\x)^\beta\,d\x\,=\,
\prod_{i=1}^n\displaystyle\left(\displaystyle
{\int_0^1 x_i^{\eta_i}(1-x_i)^{\beta_i}\,dx_i}\right),\quad \text{ for any } (\eta,\beta)\in\N^{2n},\]
and the univariate integrals
%\footnote{\tcolred{This relation already permits to get the expression for the moments of the beta distribution in (4.2) in the next section. This is why I have added  a link when discussing (4.2).}}
 may be calculated from
\begin{equation}\label{eqmoment}
\int_0^1 t^{i}(1-t)^{j}\,dt = {i!j! \over(i+j+1)!}, \quad\quad\text{ for any }  i,j \in \mathbb{N}.
\end{equation}

\subsection{The complexity of computing $f_{k}^H$ and $f_{k}^{sos}$}
\label{sec:computing procedure}
We let $N_f$ denote the set of indices  $\alpha\in \N^n$ for which $f_\alpha\ne 0$;
note that $|N_f| \le {n+d \choose d}$ if $d$ is the total degree of $f$.
The computation of
$f_{k}^H$ is done by computing the summations:
\[
\sum_{\alpha\in N_f}f_{\alpha}\,\frac{\gamma_{(\eta+\alpha,\beta)}}{\gamma_{(\eta,\beta)}}
\]
for all $(\eta,\beta) \in \N^{2n}_k$, and taking the minimum one.
(We assume that the values $\gamma_{(\eta,\beta)}$ are pre-computed for all $(\eta,\beta) \in \N^{2n}_{k+d}$.)

Thus, for fixed $(\eta,\beta) \in \N^{2n}_k$, one may first compute the inner product of the vectors with components $f_\alpha$  and $\gamma_{(\eta+\alpha,\beta)}$ (indexed by $\alpha$).
Note that these vectors are of size $|N_f|$. Since there are ${2n + k -1 \choose k}$ pairs $(\eta,\beta) \in \N^{2n}_k$, the entire computation requires
$(2|N_f|+1){2n + k -1\choose k}$ flops\footnote{We define floating point operations (flops) as in \cite[p. 18]{Golub_VanLoan};
in particular, by this definition the inner product of two $n$-vectors requires $2n$ flops.}.

As explained  in \cite{lass-siopt-10}, %mentioned before,
the computation of the upper bounds $f_{k}^{sos}$ may be done by finding the smallest  generalized eigenvalue $\lambda$ of the system:
\[
Ax = \lambda Bx \quad \quad\quad (x \neq 0),
\]
for suitable symmetric matrices $A$ and $B$ of order ${n + k \choose k}$. In particular, the rows and columns of the two matrices are indexed by $\N^{n}_{\le k}$,
and
\[
A_{\alpha, \beta} = \sum_{\delta \in N_f} f_\delta \int_{\K} \x^{\alpha + \beta + \delta} d\x, \quad B_{\alpha, \beta} = \int_{\K} \x^{\alpha + \beta} d\x \quad \alpha, \beta \in \N^{n}_{\le k}.
\]
Note that the
matrices $A$ and $B$ depend on the moments of the Lebesgue measure on $\K$, %= [0,1]^n$,
and that these moments may be computed beforehand, by assumption.
One may compute $A_{\alpha, \beta}$ by taking the inner product of $(f_\delta)_{\delta \in N_f}$ with the vector of moments $\left(\int_{\K} \x^{\alpha + \beta + \delta} d\x\right)_{\delta \in N_f}$.
Thus computation of the elements of $A$ require a total of $|N_f|\left({n + k \choose k} + 1\right)^2$ flops.

Also note that the matrix $B$ is a positive definite (Gram) matrix. Thus one has to solve a so-called symmetric-definite generalized eigenvalue problem, and this may be done
in $14{n + k \choose k}^3$ flops; see e.g.\ \cite[Section 8.7.2]{Golub_VanLoan}. Thus one may compute $f_{k}^{sos}$ in at most
  $14{n + k \choose k}^3+|N_f|\left({n + k \choose k} + 1\right)^2$ flops.

\subsection{An illustrating example}
We give an example to illustrate the behaviour of the bounds $f^{sos}_k$ and $f^H_k$. More examples will be given in  Section \ref{sec:numerical results}.
\begin{example}
\label{ex:1}
As an example we consider the bivariate
Styblinski-Tang function
$$f(x_1,x_2)=\sum_{i=1}^2 {1\over 2}(10x_i-5)^4-8(10x_i-5)^2+{5\over 2}(10x_i-5)$$
over the square $\K =[0,1]^2$, with minimum  $f_{\min, \K} \approx -78.33198$ and minimizer $$\x^* \approx (0.20906466,0.20906466).$$
Using a SOS density function,  the upper bound of degree 2 is $f_1^{sos} = -12.9249$, and the
corresponding optimal SOS density of degree $2$ is (roughly)
\[
\sigma(x_1,x_2)= (1.9169-1.005x_1-1.005x_2)^2.
\]
Using a Handelman-type density function, the upper bound of degree  $2$ is $f^H_2 = -17.3810$, with
corresponding optimal density
\begin{eqnarray*}
\sigma(x_1,x_2)=6x_2(1-x_2).
\end{eqnarray*}
On the other hand, if we consider densities of degree $6$ then we get $f_3^{sos} = -34.403$ and $f^H_6 = -31.429$.

Thus there is no general ordering between the bounds $f_k^{sos}$ and $f_{2k}^H$. Having said that, we will show in Section \ref{sec:numerical results}
that, for most of the examples we have considered, one has  $f_k^{sos} \le f_{2k}^H$ for all $k$, as one may expect from the relative computational efforts.
\begin{figure}[h!]
\begin{center}
\includegraphics[width=0.48\textwidth]{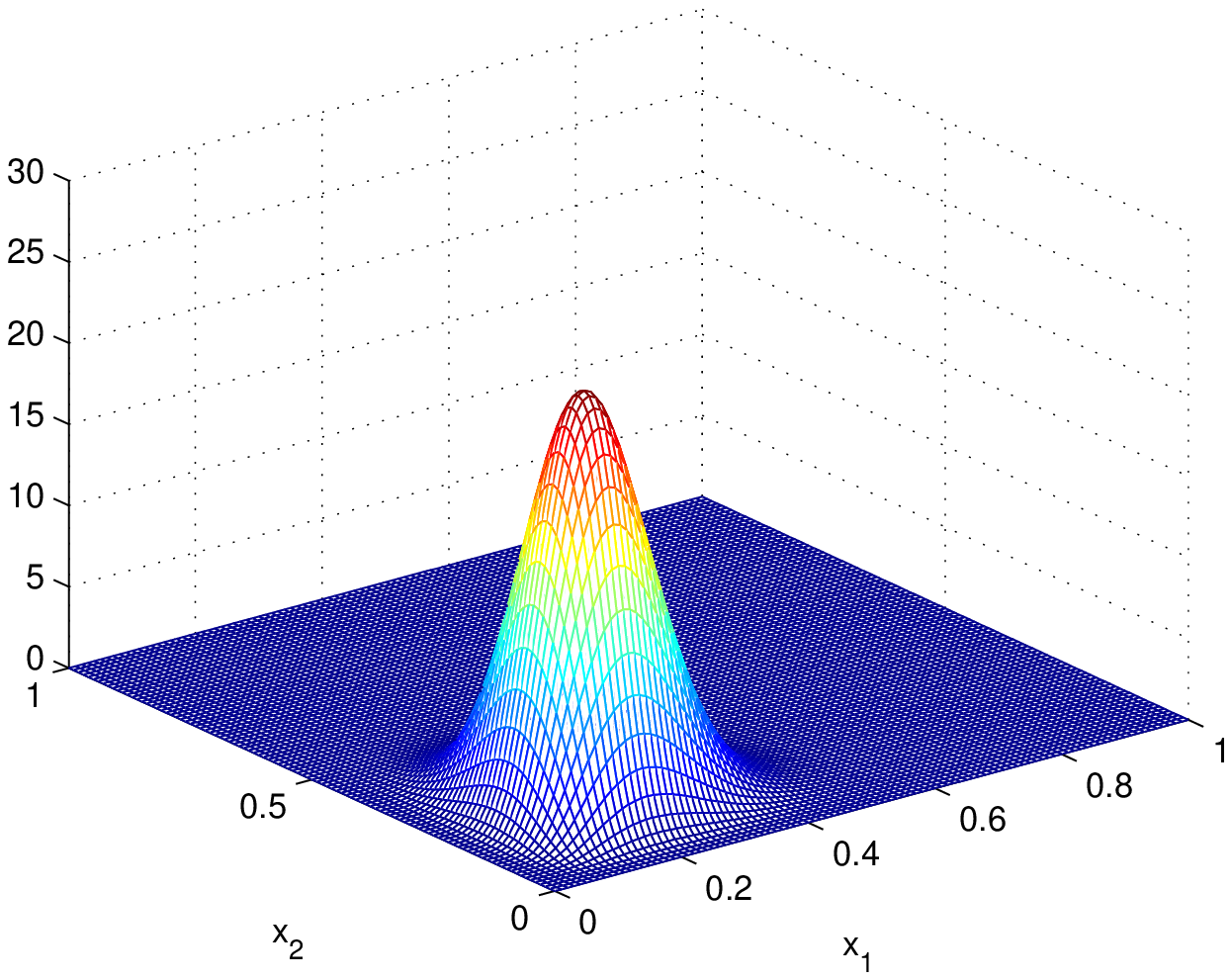} \;\;\;
\includegraphics[width=0.48\textwidth]{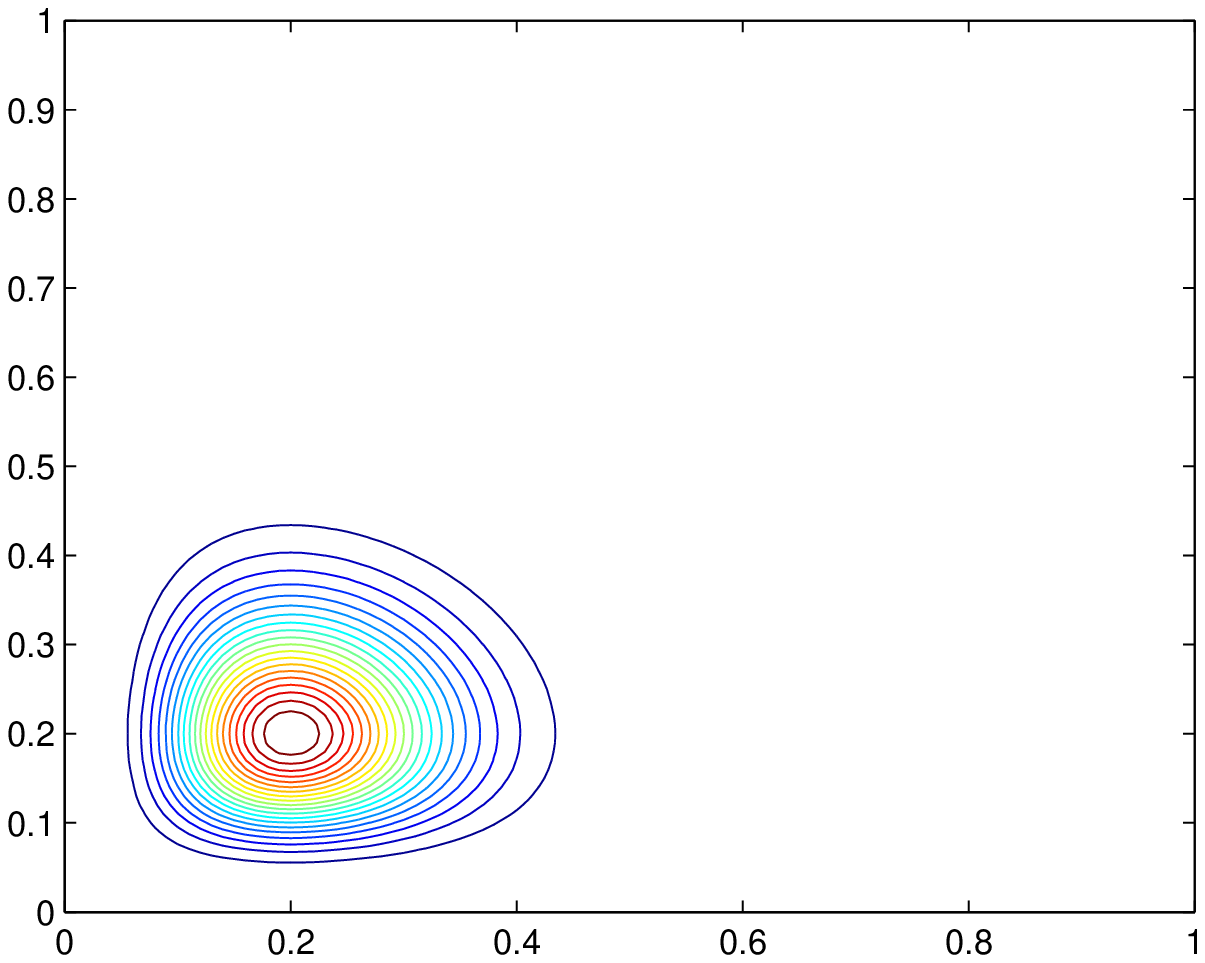}
  \caption{\label{figuresos}Optimal Handelman-type density $\sigma$ of degree $50$ on $[0,1]^2$ for the bivariate
Styblinski-Tang function.}
\end{center}
\end{figure}
As a final illustration, Figure \ref{figuresos} shows the plot and contour plot of the Handelman-type density corresponding to the bound
$f_{50}^H=-60.536$ (i.e.\ degree $50$).
%\footnote{\tcolred{What is the value of $f^H_{50}$?}}

The figure illustrates the earlier assertion that the optimal density approximates the Dirac delta
measure at the minimizer $\x^* \approx (0.20906466,0.20906466)$. Indeed, it is clear from the contour plot that
the mode of the optimal density is close to $\x^*$.
\end{example}

\section{Convergence proof for the  bounds  $f^H_k$ on $\K \subseteq [0,1]^n$}\label{secconv}

In this section we prove the convergence of the sequence
  $({f^H_k})_{k\in \N}$ to the minimum of $f$ over any  compact %semi-algebraic
set  $\K \subseteq [0,1]^n$.

\begin{theorem}
\label{th-main}
Let $\K\subseteq [0,1]^n$, let $f\in\R[\x]_d$ and let $\gamma_{(\eta,\beta)}$ be as in (\ref{gene-moments}). Define as in (\ref{maxcutformula})  the parameters
\begin{equation}
\label{main-1}
f^H_k\,=\,\displaystyle\min_{(\eta,\beta)\in\N^{2n}_k}\:\sum_{\alpha\in\N^{n}_{\le d}}f_{\alpha}\,\frac{\gamma_{(\eta+\alpha,\beta)}}{\gamma_{(\eta,\beta)}},\qquad \forall\,k\in\N.
\end{equation}
Then,  $f_{\min,\K}=\displaystyle\lim_{k\to\infty}f^H_k$.
\end{theorem}
\begin{proof}
As in (\ref{eqfsosk}), let $f_k^{sos}$ denote the bound obtained by searching over an SOS density $\sigma$ of degree at most $2k$:
\[
f_k^{sos}=\min \int_{\K} f(\x) \sigma(\x)d\x\  \text{ such that }  \int_{\K} \sigma(\x)d\x=1, \ \sigma \in \Sigma[\x]_k.
\]
Also recall from Lemma \ref{lemma:f^H reformulation} that
\[
 f^H_k= \min \int_{\K} f(\x) \sigma(\x)d\x \ \ \text{such that }  \int_{\K} \sigma(\x)d\x=1, \ \sigma \in \HH_k.
\]
By Lemma \ref{lemma:f^H reformulation}, the sequence $(f^H_k)$ is monotone non-increasing, with  $f_{\min,\K}\le   f^H_k$ for all $k$.
Hence it has a limit which is at least $f_{\min,\K}$, we now show that the limit is equal to $f_{\min,\K}$.
%sequence $(f^H_k)$  converges to $f_{\min,\K}$.

To this end, let $\epsilon >0$. As the sequence $(f^{sos}_k)$ converges to $f_{\min,\K}$ (Theorem \ref{th-lasserre}),  there exists an integer $k$ such that
$$f_{\min,\K}\le f^{sos}_k\le f_{\min,\K}+\epsilon.$$
Next, there exists a polynomial  $\sigma\in \Sigma_k$  such that $\int_{\K}\sigma(\x)d\x=1$ and
$$f^{sos}_k\le \int_{\K} f(\x)\sigma(\x) d\x \le f^{sos}_k+\epsilon.$$
Define now the polynomial $\hat\sigma(\x)=\sigma(\x)+\epsilon$. Then $\hat\sigma$ is strictly positive on $[0,1]^n$ and thus, by Theorem \ref{thm:Handelman} applied to the hypercube $[0,1]^n$,
 $\hat\sigma\in \HH_{j_k}$ for some integer $j_k$.
Observe  that $$\int_{\K} \hat\sigma(\x)d\x =\int_{\K}(\sigma(\x)+\epsilon)d\x \ge \int_{\K} \sigma(\x)d\x=1.$$
Hence we obtain:
$$ f^{H}_{j_k}-f_{\min,\K} \le {\int_{\K} f(\x) \hat\sigma(\x)d\x \over \int_{\K} \hat\sigma(\x)d\x} -f_{\min,\K}
=  {\int_{\K} (f(\x)-f_{\min,\K}) \hat\sigma(\x)d\x \over \int_{\K} \hat\sigma(\x)d\x}
\le \int_{\K} (f(\x)-f_{\min,\K})\hat\sigma(\x)d\x.$$
The right most term is equal to
$$
  \int_{\K} (f(\x)-f_{\min,\K})\sigma(\x) d\x+  \epsilon \int_{\K} (f(\x)-f_{\min,\K})d\x
  = \int_{\K} f(\x)\sigma(\x) d\x -f_{\min,\K}  +\epsilon \int_{\K} (f(\x)-f_{\min,\K})d\x,
 $$
where we used the fact that $\int_{\K}\sigma(\x) d\x =1$.
Finally, combining with the fact that $\int_{\K} f(\x)\sigma(\x)d\x\le f_{k}^{sos}+\epsilon \le f_{\min,\K}+2\epsilon$, we can derive that
$$ f^H_{j_k}-f_{\min,\K} \le \epsilon \left(2+\int_{\K} (f(\x)-f_{\min,\K})d\x\right)=\epsilon C,$$
where
$C:=2+\int_{\K} (f(\x)-f_{\min,\K})d\x$  is a constant. This concludes the proof.
\end{proof}
Note that, in the proof, it was essential to have $\hat\sigma$ strictly positive on all of $[0,1]^n$, for the application of Handelman's theorem.
The fact that $\hat\sigma(\x)=\sigma(\x)+\epsilon$ with $\sigma$ SOS and $\epsilon > 0$ guaranteed this strict positivity.

\section{Bounding the rate of convergence for the  parameters  $f^H_k$ on $\K = [0,1]^n$}\label{secerror}

In this section we analyze the rate of convergence  of the bounds  $f^H_{k}$ for  the hypercube $\K=[0,1]^n$.
We prove a convergence rate in $O(1/\sqrt k)$ for the range $f^H_{k}-f_{\min,\K}$ in general, and a stronger convergence rate
in $ O(1/k)$ when $f$ has a rational global minimizer in $[0,1]^n$, which is the case, for instance, when $f$ is quadratic.

% and $f$ has a rational global minimizer on $\K$.
Our main tool will be exploiting some properties of the moments $\gamma_{(\eta,\beta)}$ which, as we recall below,
 arise from the moments of the beta distribution.
%\footnote{\tcolred{I modified a bit this sentence: indeed we already can derive the form of the moments as indicated in the next section and they are anyway easy to compute directly; the point is to make clear the link to the classical beta distribution?}}
%Our main tool will be the (moments of the) beta distribution, and we first review the necessary details.

\subsection{Properties of the beta distribution}\label{secbeta}
\noindent
By definition, a random variable $X \in [0,1]$ has the beta distribution with shape parameters
$a >0$ and $b>0$, which is denoted by $X\sim beta(a,b)$, if its probability density function is given by
\begin{equation*}\label{density}
y\mapsto \frac{y^{a-1}(1-y)^{b-1}}{\int_{0}^{1}t^{a-1}(1-t)^{b-1}dt}.
\end{equation*}
If $a>1$ and $b>1$, then the (unique) mode of the distribution (i.e.,\ the maximizer of the density function) is
\begin{equation}
\label{eq:mode}
y =  (a-1)/(a+b-2).
\end{equation}
Moreover, the $k$-th moment of $X$ is given by
\begin{equation}\label{kmoment}
\mathbb{E}(X^k)={a(a+1)\cdots(a+k-1) \over (a+b)(a+b+1)\cdots(a+b+k-1)}, \quad (k = 1,2,3,\ldots)
\end{equation}
(see, e.g, \cite[Chapter 24]{BL09}; this also follows using (\ref{eqmoment})).

\medskip
In what follows we will consider families of
 random random variables with the beta distribution of the form
$X  \sim beta(ar,br)$, where $a$ and $b$ are positive real numbers and $r \ge 1$ is an integer.
By (\ref{kmoment}), any such  random variable has mean
\[
\mathbb{E}(X) = \frac{ar}{ar+br} = \frac{a}{a+b}.
\]

%The next result shows how the moments of such random variables relate to powers of the mean.
In Lemma \ref{lemmomdif0} below we show how the moments of such random variables relate to powers of the mean.
The proof relies on the following technical lemma.

%The proof of this result is an immediate consequence of \eqref{kmoment} and the following technical lemma.

\begin{lemma}\label{lemapprox}
Let  $k$ be a positive integer.
There exists a constant $C_k>0$ (depending only on $k$) for which the following relation holds:
\begin{equation}\label{eqCk}
{ rp(rp+1)\cdots (rp+k-1)\over rq(rq+1) \cdots (rq+k-1)} -{p^k\over q^k} \le {C_k\over r}
\end{equation}
for all integers  $r\ge 1$, and real numbers $0< p<q$.
\end{lemma}

\begin{proof}
Consider the univariate polynomial $\phi(t)=(t+1)\cdots (t+k-1) =\sum_{i=0}^{k-1}a_it^i$, where the scalars $a_i>0$ depend only on $k$ and $a_{k-1}=1$.
Denote by $\Delta$ the left hand side in (\ref{eqCk}),
which can be written as $\Delta=N/D$, where we set
$$N:= rpq^k \phi(rp)-rqp^k\phi(rq),\ \ \ D:= rq^{k+1} \phi(rq).$$
We first work out the term $N$:
$$N= rpq \left(\sum_{i=0}^{k-2} a_ir^ip^iq^{k-1} -\sum_{i=0}^{k-2}a_ir^iq^ip^{k-1}\right)
= rpq \sum_{i=0}^{k-2} a_ir^i p^iq^i(q^{k-1-i}-p^{k-1-i}).$$
Write: $q^{k-1-i}-p^{k-1-i} =(q-p)\sum_{j=0}^{k-2-i} q^jp^{k-2-i-j}\le (q-p)q^{k-2-i}(k-1-i)$, where we use the fact that $p<q$.
This implies:
$$N\le  rpq (q-p) \sum_{i=0}^{k-2}a_ir^ip^iq^{k-2} (k-1-i) = rpq^{k-1}(q-p) \sum_{i=0}^{k-2} a_i(k-1-i) r^ip^i=:N'.$$
%where we have used again that $p<q$.
Thus we get:
$$\Delta \le {N'\over D}= {p(q-p)\over q^2} \cdot {\sum_{i=0}^{k-2}a_i(k-1-i) r^i p^i\over\phi(rq)}.$$
The first factor is at most 1, since  one has: $p(q-p)\le q^2$, as $q^2-p(q-p)=(q-p)^2+pq$.
Second, we bound the sum $\sum_{i=0}^{k-2} a_i(k-1-i) r^i p^i$ in terms of $\phi(rq)= \sum_{j=0}^{k-1}a_jr^j q^j.$
Namely, define the constant
$$C_k:=\max_{0\le i\le k-2} {a_i(k-1-i)\over a_{i+1}},$$
which depends only on $k$. We show that
$$a_i (k-1-i)r^ip^i \le {C_k \over r}.$$
Indeed, for each $0\le i\le k-2$,  using $p^i \le q^{i+1}$ and the definition of $C_k$, we get:
$$r \cdot a_i(k-1-i)r^ip^i \le a_i (k-1-i)r^{i+1}q^{i+1} \le C_k a_{i+1}r^{i+1}q^{i+1}.$$
% \le C_k D', $$
%since each term indexed by $j\ne i+1$ in the summation defining $D'$ is nonnegative.
Summing over $i=0,1,\ldots,k-2$ gives:
$$r\sum_{i=0}^{k-2} a_i(k-1-i)r^ip^i\le C_k \sum_{i=0}^{k-2} a_{i+1}r^{i+1}q^{i+1} \le C_k \phi(rq),$$
and thus
$$\Delta \le {N'\over D}\le {C_k\over r}$$
as desired.
\end{proof}

\begin{lemma}\label{lemmomdif0}
%Let $r$ be a positive integer, let $a$ and $b$ be positive real numbers, and $X  \sim beta(ar,br)$.
For any integer $k\ge1$, there exists a constant $C'_k>0$ (depending only on $k$) for which the following holds:
\begin{equation*}
\left|\mathbb{E}(X^k)-(\mathbb{E}(X))^k\right| \le {C'_k\over r},
%O\left({1\over r}\right).
\end{equation*}
for all integers $r\ge 1$, real numbers $a,b>0$, and where $X \sim beta(ar,br)$.
\end{lemma}

\begin{proof}
Directly using (\ref{kmoment}),  $\mathbb E(X)={a\over a+b}$,  and Lemma \ref{lemapprox} applied to $p=a$ and $q=a+b$.
\end{proof}

Now we consider i.i.d.\ random variables $X_1,\ldots,X_n$ such that
\begin{equation}
\label{beta iid}
X_i  \sim beta(a_ir,b_ir) \quad a_i,b_i > 0 \; (i \in [n]), \; r \ge 1, r \in \mathbb{N},
\end{equation}
and denote $X = (X_1,\ldots,X_n)$.
For given $\alpha \in \mathbb{N}^n$, we denote $X^\alpha = \prod_{i=1}^n X_i^{\alpha_i}$.
Since the random variables $X_i$ are independent we have $\mathbb{E}(X^\alpha) = \prod_{i=1}^n \mathbb{E}(X_i^{\alpha_i})$ and, for a polynomial $f=\sum f_\alpha \x^\alpha$,   the
 expected value of $f(X)=\sum_{\alpha\in\oN^n}f_{\alpha}X^{\alpha}$ is given by
\begin{equation}\label{expf}
\mathbb{E}(f(X))=\sum_{\alpha\in\oN^n}f_{\alpha}\mathbb{E}(X^{\alpha})=\sum_{\alpha\in\oN^n}f_{\alpha}\prod_{i=1}^n\mathbb{E}(X_i^{\alpha_i}).
\end{equation}
Recall that  the explicit value of $\mathbb{E}(X_i^{\alpha_i})$ is given by \eqref{kmoment}.
The next result relates $\mathbb{E}(f(X))$ (the expected value of $f(X)$) and $f(\mathbb E(X))$  (the value of $f$ evaluated at the mean of $X$).

\begin{lemma}
\label{lemma:fbound}
Let $f(\x)=\sum_{\alpha\in\oN^n}f_{\alpha}\x^{\alpha}$ and $X = (X_1,\ldots,X_n)$, where the i.i.d.\ random variables $X_i$ $(i \in [n])$ are
as in \eqref{beta iid}. Then there is a constant $\hat C_f>0$ (depending on $f$ only) such that
\[
|\mathbb{E}(f(X))-f\left( \mathbb{E}(X)  \right)|\le \frac{\hat C_f}{r}.
\]
\end{lemma}

\begin{proof}
 We have
\begin{eqnarray*}
\mathbb{E}(f(X))-f(\mathbb{E}(X))
=\sum_{\alpha\in\oN^n}f_{\alpha}{\left(\prod_{i=1}^n\mathbb{E}(X_i^{\alpha_i})-\prod_{i=1}^n(\mathbb{E}(X_i))^{\alpha_i}\right)}.
\end{eqnarray*}

\noindent
By the identity:
\begin{equation}\label{eqxy}
\prod_{i=1}^n x_i -\prod_{i=1}^n y_i=\sum_{i=1}^n \left[(x_i-y_i)\prod_{j=1}^{i-1}y_j \prod_{j=i+1}^{n}x_j\right] \quad \quad (\x,\y \in \mathbb{R}^n),
\end{equation}
one has
\begin{eqnarray*}
\prod_{i=1}^n\mathbb{E}(X_i^{\alpha_i})-\prod_{i=1}^n(\mathbb{E}(X_i))^{\alpha_i}=\sum_{i=1}^n \left( \left(\mathbb{E}(X_i^{\alpha_i})-(\mathbb{E}(X_i))^{\alpha_i}\right){\prod_{j=1}^{i-1}(\mathbb{E}(X_j))^{\alpha_j}} {\prod_{j=i+1}^{n}\mathbb{E}(X_j^{\alpha_j})}\right).
\end{eqnarray*}

\smallskip\noindent
Since $\mathbb{E}(X_i)\in[0,1]$  and $\mathbb{E}(X_i^{\alpha_i})\in [0,1]$ for any $i\in[n]$, it follows that
\begin{eqnarray*}
|\mathbb{E}(f(X))-f(\mathbb{E}(X))| &\le& \sum_{\alpha\in\oN^n}\left|f_{\alpha}\right| \sum_{i=1}^n \left|\mathbb{E}(X_i^{\alpha_i})-(\mathbb{E}(X_i))^{\alpha_i}\right| \\
                                   &\le & \sum_{\alpha\in\oN^n}\left|f_{\alpha}\right| \sum_{i=1}^n {C'_{\alpha_i} \over r},
\end{eqnarray*}
where the second inequality is from Lemma \ref{lemmomdif0},
and the constant $C'_{\alpha_i} > 0$ only depends on $\alpha_i$.
Setting $\hat C_f:=\sum_{\alpha\in\N^n} |f_\alpha| \sum_{i=1}^nC'_{\alpha_i}$
 concludes the proof.
\end{proof}

\subsection{Proof of the convergence rate}

Let $\x^*$ be a global minimizer  of $f$  in $[0,1]^n$.
Our objective is to analyze the rate of convergence of the sequence $(f^H_k-f(\x^*))_k$.
Our strategy is to define  suitable shape parameters $\eta^*_i,\beta^*_i$ from the components $x^*_i$ of the global minimizer $\x^*$
so that, if we choose a vector $X=(X_1,\ldots,X_n)$ of i.i.d. random variables  with $X_i\sim beta(\eta^*_i,\beta^*_i)$, then (roughly)
$\mathbb E(X)\approx \x^*$ and $\mathbb E(f(X))\approx f^H_k$ (so that we can use the result of Lemma \ref{lemma:fbound} to estimate $f^H_k-f(\x^*)$.

\medskip
In a first step we indicate how to define the shape parameters $\eta^*_i, \beta^*_i$. For any given integer $r\ge 1$ we will select them of the form $\eta^*_i=ra_i$, $\beta^*_i=rb_i$, where $a_i,b_i$  are constructed from the coordinates of $\x^*$.
%We will select them of the form $\eta_i=r a_i$, $\beta_i=r b_i$, where $r\ge 1$ is a given fixed  integer and $a_i,b_i$ are integers constructed from the components of $\x^*$.
As we want $\eta^*_i,\beta^*_i$  to be integer valued  we need to discuss whether a coordinate $x_i$ is rational or not, and to deal with  irrational coordinates we will use the following result about Diophantine approximations.

%In what follows we indicate how to construct   a vector of independent random variables $X = (X_1,\ldots,X_n)$ so that
%the $X_i$'s have the beta distribution with suitable shape parameters $\eta_i^*,\beta_i^*$, designed to ensure that  (roughly)  $\mathbb{E}[X] = \x^*$.

%We will use the following result about Diophantine approximations.

\begin{theorem}[Dirichlet's theorem] (See e.g. \cite[Chapter 6.1]{Schrijver})\label{theoDir}
Consider a real number $x\in \R$  and $0<\epsilon\le 1$.
Then there exist integers $p$ and $q$ satisfying
\begin{equation*}\label{eqDio}
\left|x-{p\over q}\right| <{\epsilon\over q} \ \text{ and } \ 1\le q\le {1\over \epsilon}.
\end{equation*}
If $x\in (0,1)$,  then one may moreover assume  $0\le p\le q$.
\end{theorem}

\begin{definition}[Shape parameters for rational components]\label{defbetaparaQ}
Fix an integer $r\ge 1$. For rational coordinates $x^*_i\in\mathbb Q$ define $\eta_i^*$, $\beta^*_i$ as follows:
%Set $J=\{i\in [n]: x_i\in \mathbb Q\cap (0,1)\}$.
%For any $i\in J$, write $x_i^*=p_i/q_i$, where $p_i,q_i$ are integers such that
%$1\le p_i<q_i$.
%We define the parameters $\eta_i^*$ and $\beta_i^*\in\oN$ as follows.
\begin{itemize}
\item[(i)]
If $x^*_i=0$ then set $\eta^*_i=1$ and $\beta^*_i=r$.
\item[(ii)]
If $x_i^*=1$ then set $\eta^*_i=r$ and $\beta^*_i=1$.
\item[(iii)]
If $x^*_i\in \mathbb Q\setminus \{0,1\}$ then write $x^*_i=p_i/q_i$ where $1\le p_i<q_i$ are integers, and  set $\eta^*_i=r p_i$ and $\beta^*_i=r(q_i-p_i)$.\end{itemize}
%\textcolor{red}{Set $J=\{i\in [n]: x_i\in \mathbb Q\setminus \{0,1\}\}$.}
\end{definition}

\ignore{
If $x_i^*\in (0,1)$ is a rational coordinate of $\x^*$, then we select integers $p_i$ and $q_i$ such that $x_i^*=p_i/q_i$, so that
$1\le p_i<q_i$.
When $x_i^*$ is an irrational coordinate of $\x^*$ we use Theorem \ref{theoDir} to construct a pair of suitable integers $p_i,q_i$.
%We will use this  result to construct the parameters $\eta_i^*$ and $\beta_i^*$ corresponding to the irrational coordinates $x_i^*$ ($x_i^*\in(0,1)$) of the minimizer $\x^*$.
Namely, we consider an integer $r\ge 1$ and apply Theorem \ref{theoDir} with $\epsilon=1/r$. Then,  there exist  integers $p_i$ and $q_i$ satisfying
\begin{equation}\label{eq2}
\left|x_i^*-{p_i\over q_i}\right| <{1\over rq_i}, \ 0\le p_i\le q_i\le r \  \text{and}\ 1\le q_i.
\end{equation}
For convenience, let $I_0$ (resp., $I_1$, $I$) denote the set of indices $i\in [n]$ for which
$x_i^*$ is irrational and the integers $p_i$ and $q_i$ in (\ref{eq2}) satisfy: $p_i=0$ (resp., $p_i=q_i$, $1\le p_i<q_i$).
Moreover, define the set $J$ consisting
%$J:=[n]\setminus (I_0\cup I_1\cup I)$
\tcolred{ of all indices $i$ for which $x_i^* \in (0,1)$ is rational. Then, $x^*_i\in \{0,1\}$ for all $i\in [n]\setminus (I_0\cup I_1\cup I \cup J)$.}

We now indicate how to construct the parameters $\eta_i^*$ and $\beta_i^*$.
}

\begin{definition}[Shape parameters for irrational components]\label{defbetaparaNQ}
Fix an  integer $r\ge 1$.
For each irrational coordinate $x^*_i \in \R\setminus \mathbb Q$,  apply Theorem \ref{theoDir} with $\epsilon=1/r$ to obtain integers $p_i,q_i$ satisfying
\begin{equation*}\label{eq2}
\left|x_i^*-{p_i\over q_i}\right| <{1\over rq_i}, \ 0\le p_i\le q_i\le r, \  \text{and}\ 1\le q_i.
\end{equation*}
Define the sets
$I_0=\{i\in [n]: x_i^*\in \R\setminus \mathbb Q, \ p_i=0\},\
I_1=\{i\in [n]: x_i^*\in \R\setminus \mathbb Q,\  p_i=q_i\},$ and
$I=\{i\in [n]: x_i^*\in \R\setminus \mathbb Q, \ 1\le p_i<q_i\}$,
and define $\eta^*_i$, $\beta^*_i$ as follows:
\begin{itemize}
\item[(iv)]
If $i\in I_0$ then set $\eta^*_i=1$ and $\beta^*_i=r$.
\item[(v)]
If $i\in I_1$ then set $\eta^*_i=r$ and $\beta^*_i=1$.
\item[(vi)]
If $i\in I$ then set $\eta^*_i=rp_i$ and $\beta^*_i=r(q_i-p_i)$.
\end{itemize}
% For each coordinate $x_i^*\in[0,1]$, consider the integers $p_i$ and $q_i$ defined as above.
%We define the parameters $\eta_i^*$ and $\beta_i^*\in\oN$ as follows.
%\begin{itemize}
%\item[(i)]
%Assume $i\in J\cup I$; that is, either $x_i^*\in(0,1)$ is  rational of the form    $x_i^*=p_i/q_i$,
%or  $x_i^*$ is irrational with    $1\le p_i<q_i$.
%Then,
%we set $\eta_i^*=rp_i$ and $\beta_i^*=r(q_i-p_i)$.
%\item[(ii)]
%Assume either $x_i^*=0$,  or $i\in I_0$, i.e.,  $x_i^*$ is irrational with  $p_i=0$.
%Then we set  $\eta_i^*=1$ and $\beta_i^*=r$.
%\item[(iii)] Assume  either $x_i^*=1$, or $i\in I_1$, i.e., $x_i^*$ is irrational with  $p_i=q_i$.
%Then we set $\eta_i^*=r$ and $\beta_i^*=1$.
%\end{itemize}
%Now we define the vector of independent random variables $X := (X_1,\ldots,X_n)$, where $X_i \sim beta(\eta_i^*,\beta_i^*)$ $(i \in [n])$.
\end{definition}

As above consider i.i.d. $X=(X_1,\ldots,X_n)$, where $X_i\sim beta(\eta^*_i,\beta^*_i)$.
Then, by construction, for all $i \in [n]$, one has
\[
\mathbb{E}(X_i) = \frac{\eta_i^*}{\eta_i^*+\beta_i^*} = \left\{
\begin{array}{ll}
 \frac{1}{r+1} & \mbox{in cases (i), (iv),} \\
\frac{r}{r+1}  & \mbox{in cases (ii), (v),} \\
\frac{p_i}{q_i} & \mbox{in cases (iii), (vi).} \\
\end{array}\right.
\]
One can verify that  in all cases one has
\begin{equation}\label{eqXi}
|\mathbb{E}(X_i)- x_i^*| \le 1/r\ \text{ for all } i\in [n].
\end{equation}
Observe morever that, again by construction,
\begin{equation}\label{eqbeta}
\mathbb{E}(f(X)) ={\int_{[0,1]^n} f(\x)\x^{\eta^* - \mathbf{1}} (\1-\x)^{\beta^*-\mathbf{1}}d\x\over \int_{[0,1]^n} \x^{\eta^* - \mathbf{1}} (\1-\x)^{\beta^*-\mathbf{1}}d\x}  \ge f^H_{k_r} \ge f(\x^*)\ ,
\end{equation}
 where we let $\mathbf{1}$ denote the all-ones vector and we define the parameter
\begin{equation}\label{eqkr}
 k_r:=\sum_{i=1}^n (\eta_i^*-1+\beta_i^*-1).
 \end{equation}
We will  use  the following estimate on the parameter $k_r$.

\begin{lemma}\label{lemkr}
Consider  the parameter $k_r=\sum_{i=1}^n (\eta_i^*-1+\beta_i^*-1)$ {and  $J=\{i\in [n]: x^*_i\in\mathbb Q\setminus \{0,1\}\}$.}
Then the following holds:
\begin{itemize}
\item[(a)] If $\x^*\in \mathbb Q^n$ then $k_r\le ar$ for all $r\ge 1$, where $a>0$ is a  constant  (not depending on $r$).
\item[(b)]
If $\x^*\not \in  \mathbb Q^n$ then  $k_r\le a'r^2$ for all $r\ge 1$, where $a'>0$ is a constant (not depending on $r$).
\item[(c)]
For $r=1$, we have that $k_1= \sum_{i\in J} q_i -2|J|$. %, where $J=\{i\in [n]: x^*_i\in\mathbb Q\setminus \{0,1\}\}$.
\end{itemize}
\end{lemma}

\begin{proof}
By construction, $\eta_i^*+\beta_i^*-2= rq_i-2$ for each $i\in I\cup J$, and $\eta_i^*+\beta_i^*-2=r-1$ otherwise.
From this one gets $k_r= r(\sum_{i\in I\cup J}q_i +n-|I\cup J|) -n-|I\cup J|=: ar-b$, after setting
$b:= n+|I\cup J|$ and
$a:=\sum_{i\in I\cup J}q_i +n-|I\cup J|$, so that $a,b\ge 0$. Thus, $k_r\le ar$ holds.

Next, note that $q_i\le r$ for each $i\in I$, while $q_i$ does not depend on $r$ for $i\in J$ (since then $x_i^*=p_i/q_i$).
Hence, in case  (a),  $I=\emptyset$  and the constant $a$ does not depend on $r$.
In case (b),  we obtain: $a\le r|I| + \sum_{i\in J}q_i +n-|I\cup J| \le a'r$, after setting
$a':= |I|+ \sum_{i\in J}q_i +n-|I\cup J|$, which is thus a constant not depending on $r$.
Then, $k_r\le a r\le a'r^2$.

\tcolred{In  the case   $r=1$,
 the set $I$ is empty and thus $k_1=\sum_{i\in J}q_i -2|J|$, showing (c).}
\end{proof}

\smallskip\noindent
We can now prove the following upper bound for the range $\mathbb{E}(f(X))-f(\x^*)$ (thus also for the range $f^H_{k_r}- f(\x^*)$)
which will be crucial for establishing the rate of convergence of the parameters $f^H_k$.

\begin{theorem}\label{thmdif}
Given a polynomial $f$ of total degree $d$,  consider a global minimizer $\x^*$ of $f$  in $[0,1]^n$.
 Let $r$ be a positive integer. For any $x_i^*\in[0,1]$ ($i\in[n]$), consider the parameters $\eta_i^*,\beta_i^*$ as in Definitions \ref{defbetaparaQ} and \ref{defbetaparaNQ}, and i.i.d. random variables $X_i\sim beta(\eta^*_i,\beta^*_i)$.  Then there exists a constant $C_f>0$ (depending only on $f$)  such that
\begin{eqnarray*}
f^H_{k_r}-f(\x^*)\le \mathbb{E}(f(X))-f(\x^*) \le {C_f\over r},
%=O\left({1\over r}\right).
\end{eqnarray*}
where $k_r$ is as in (\ref{eqkr}).
\end{theorem}

\begin{proof}
\noindent
The leftmost inequality follows using (\ref{eqbeta}), we show the rightmost one.
By Lemma \ref{lemma:fbound} one has:
\begin{eqnarray*}
\mathbb{E}(f(X))-f(\x^*) &= & \mathbb{E}(f(X))-f(\mathbb{E}(X)) +  f(\mathbb{E}(X))-f(\x^*) \\
                         &\le& \hat C_f/r + f(\mathbb{E}(X))-f(\x^*),
\end{eqnarray*}
 where $\hat C_f >0$ is a constant that depends on $f$ only. Thus we need only bound $ f(\mathbb{E}(X))-f(\x^*)$.
 To this end, note that
\begin{eqnarray*}
f(\mathbb{E}(X))-f(\x^*)
&=& \sum_{\alpha\in\oN^n}f_{\alpha}{\left(\prod_{i=1}^n\mathbb{E}(X_i)^{\alpha_i}-\prod_{i=1}^n(x^*_i)^{\alpha_i}\right)}.
\end{eqnarray*}
Using again  the identity (\ref{eqxy})
%\begin{equation*}
%\prod_{i=1}^d x_i -\prod_{i=1}^d y_i=\sum_{i=1}^d \left[(x_i-y_i)\prod_{j=1}^{i-1}y_j \prod_{j=i+1}^{d}x_j\right] \quad \quad (x,y \in \mathbb{R}^d),
%\end{equation*}
one has
\[
\left|\left(\prod_{i=1}^n\mathbb{E}(X_i)^{\alpha_i}-\prod_{i=1}^n(x^*_i)^{\alpha_i}\right)\right| \le \sum_{i:\alpha_i > 0} |\mathbb{E}(X_i)-x_i^*| \le \frac{d}{r},
\]
where $d$ is the degree of $f$, and we have used $| \mathbb{E}(X_i)-x_i^*| \le 1/r$, $x_i^* \in [0,1]$ and  $\mathbb{E}(X_i) \in [0,1]$ for all $i \in [n]$.
Setting
\[
C_f = \hat C_f + {d\sum_{\alpha\in\oN^n}|f_{\alpha}|}
\]
completes the proof.
\end{proof}

\medskip

Finally we can now show the following  for  the rate of convergence of the sequence $f^H_k$, which is our main result.

\begin{theorem}
\label{cor:convergence rate}
Let $f$ be  a polynomial, % let $f_{\min,\K}$ be its minimum over the hypercube $\K=[0,1]^n$,
let $\x^*$ be  a global minimizer  of $f$  in $[0,1]^n$,
and consider as before the parameters
\[f^H_k=\,\min_{(\eta,\beta)\in \N_k^{2n}}\:\frac{\int_{[0,1]^n} f(\x)\,\x^\eta(\1-\x)^\beta\,d\x}
{\int_{[0,1]^n} \x^\eta(\1-\x)^\beta\,d\x} \quad (k = 1,2,\ldots).
\]
\tcolred{There exists a constant $M_f$ (depending only on $f$) such that
\begin{equation}\label{rateNQ}
f^H_k - f(\x^*)  \le {M_f \over \sqrt k} \ \ \text{ for all } k\ge k_1,
\end{equation}
where $k_1=\sum_{i\in J}q_i -2|J|$, %(as in Lemma \ref{lemkr} (c)).  
with $J=\{i\in [n]: x^*_i\in \mathbb Q\setminus \{0,1\}\}$
and $x_i^*=p_i/q_i$ for integers $1\le p_i<q_i$ if $i\in J$.
 %with $\x^*$ a global minimizer of $f$ in $[0,1]^n$.} %(as in Lemma \ref{lemkr} (iii)).
%$f^H_k - f(\x^*) = O(1/\sqrt{k})$.
Moreover,
if $f$ has at least one  rational global minimizer $\x^*$,
then there exists a constant $M'_f$  (depending only on $f$) such that
\begin{equation}\label{rateQQ}
f^H_k - f(\x^*)  \le {M'_f \over  k} \ \ \text{ for all } k\ge k_1.
\end{equation}
}
In particular, the convergence rate is in $O(1/k)$ when $f$ is a quadratic polynomial.
\end{theorem}

\begin{proof}
%\textcolor{red}{Let $\x^*$ be a global minimizer of $f$ in $\K=[0,1]^n.$}
%By Theorem \ref{thmdif}, for any  integer $r\ge 1$ and the corresponding parameter $k_r$ (from Lemma \ref{lemkr}),
  %and, as in Definitions \ref{defbetaparaQ} and \ref{defbetaparaQ},  the parameters $\eta_i^*, \beta_i^*$ and
%  the  random variables {$X_i\sim beta(\eta_i^*,\beta_i^*)$} for $i\in [n]$. Recall also the parameter $k_r$ from Lemma \ref{lemkr}.
%Then, as observed earlier in (\ref{eqbeta}), by the definition of the parameter $f^H_{k_r}$, we have that $\mathbb{E}(f(X))\ge f^H_{k_r}$.
%$ \mathbb{E}(f(X))-f(\x^*) \le C_f/r$ for some  constant $C_f$ (depending only on $f$).
% we have the  inequality:
%\begin{equation}
%\label{eq:frk rate}
%f^H_{k_r}-f(\x^*) \le {C_f\over r}.
%\end{equation}
Consider   an arbitrary integer $k\ge k_1$. Let $r\ge 1$ be the largest integer for which $k\ge k_r$. Then we have $k_r\le k<k_{r+1}$.
% in the interval $[k_r, k_{r+1}]$.
 As $k_r\le k$, we have the inequality
$f^H_k \le f^H_{k_r} $ and thus, by Theorem~\ref{thmdif}, we obtain
% \tcolred{using (\ref{eq:frk rate}),} $f^H_{k}-f(\x^*) \le {C_f\over r}.$ We now bound $1/r$ in terms of $k$.
$$ f^H_k-f(\x^*)\le f^H_{k_r}-f(\x^*) \le {C_f\over r},$$
where the constant $C_f$ depends only on $f$.

%\smallskip
If $\x^*\in \mathbb Q^n$ then, by Lemma \ref{lemkr} (a), $k_{r+1}\le a(r+1)\le 2ar$. This  implies
$k\le k_{r+1} \le 2ar$, where the constant $a$ does not depend on $r$.
Thus, $$f^H_{k}-f(\x^*) \le {C_f\over r} \le {2aC_f\over k} \tcolred{= {M_f\over k}},$$
 \tcolred{where the constant $M_f={2aC_f}$ depends only on $f$. This shows (\ref{rateQQ}).}
%which shows that $f^H_{k}-f(\x^*)=O(1/k)$ as desired.

%\smallskip
If $\x^*\not\in \mathbb Q^n$ then, by Lemma \ref{lemkr} (b), $k_{r+1}\le a'(r+1)^2 \le 4a'r^2$. This  implies
$k\le k_{r+1} \le 4a' r^2$ and thus ${1\over r}\le {2\sqrt {a'}\over \sqrt k}$, where the constant $a'$ does not depend on $r$.
Therefore,  $$f^H_{k}-f(\x^*) \le {C_f\over r} \le {2\sqrt{a'}C_f\over \sqrt k}=  {M'_f\over \sqrt k},$$
where the constant $M'_f= 2\sqrt{a'}C_f$ depends only on $f$. This  shows
% \tcolred{$f^H_{k}-f(\x^*)\le {M'_f\over \sqrt k}$ and thus
(\ref{rateNQ}).
% after setting $M'_f= 2\sqrt{a'}C_f$.}

Finally, if  $f$ is quadratic then, by a result of  Vavasis \cite{Vavasis},  $f$ has a rational minimizer over the hypercube and thus the rate of convergence is  $O(1/k)$.
\end{proof}

\tcolred{Note that the inequalities (\ref{rateNQ}) and (\ref{rateQQ}) hold for all $k\ge k_1$, where $k_1$ depends only on the rational components in $(0,1)$ of the minimizer
$\x^*$. The constant $k_1$ can be in $O(1)$, e.g., when all but $O(1)$ of these rational components have a small denominator (say, equal to $2$).
Thus we can, for some problem classes, get a bound with an error estimate in polynomial time.}

\begin{example}\label{extight}
Consider the polynomial $f=\sum_{i=1}^n x_i$ and the set $\K=[0,1]^n$. Then $f_{min,\K}=0$ is attained at $\x^*=0$.
Using
%\eqref{maxcutformula} together with
the relations  \eqref{gene-moments}, \eqref{eqmoment} and  \eqref{main-1},  it follows that
$$f_k^H=\min_{(\eta,\beta)\in\N^{2n}_k} \sum_{i=1}^n {\eta_i+1\over \eta_i+\beta_i+2}.$$
Since $\eta_i+\beta_i\le k$ and $\eta_i\ge 0$ (for any $i\in[n]$), we have $f_k^H\ge {n\over k+2}$.

\smallskip
\noindent
By this example, there does not exist any $\delta>0$ such that, for any $f$, $f_k^H-f_{\min,\K}=O(1/k^{1+\delta}).$
Therefore, when a rational minimizer exists, the convergence rate from Theorem \ref{cor:convergence rate} in $O(1/k)$ for $f_k^H$  is tight.
\end{example}

%\subsection{Comparison to minimization on a regular grid}\label{secgrid}
\section{Bounding the rate of convergence for grid search over $\K=[0,1]^n$}\label{secgrid}

As an alternative to computing $f^H_k$ on $\K = Q := [0,1]^n$, one may minimize $f$ over the regular grid:
\[
Q(k):=\{\x\in Q=[0,1]^n \mid k\x\in \N^n\},
\]
i.e.,\ the set of rational points in $[0,1]^n$ with denominator $k$.
Thus we get the upper bound
%\begin{equation}\label{relpQd}
\[
f_{\min,Q(k)}:=\min _{\x\in Q(k)} f(x) \ge f_{\min,Q} \quad k = 1,2,\ldots
%\end{equation}
\]
De Klerk and Laurent \cite{KL10} showed  a rate of convergence in $O(1/k)$ for this sequence  of upper bounds:
\begin{equation}\label{eqKL}
f_{\min,Q(k)}-f_{\min,Q} \le  {L(f)\over k} {d+1\choose 3}n^d \ \text{ for any } k\ge d,
\end{equation}
where $d$ is the degree of $f$ and $L(f)$ is the constant
\[
L(f) =\max _\alpha |f_\alpha| {\prod_{i=1}^n \alpha_i ! \over |\alpha|!}.
\]
We can in fact show a stronger convergence rate in $O(1/k^2)$.

{
\begin{theorem}
Let $f$ be a polynomial and let $\x^*$ be a global minimizer of $f$ in $[0,1]^n$.
Then there exists a constant $C_f$ (depending on $f$) such that
\[
f_{\min,Q(k)} -f(\x^*) \le {C_f\over k^2} \ \ \text{ for all } k\ge 1.
\]
\end{theorem}
}

\begin{proof}
{
Fix $k\ge 1$. By looking at the grid point in $Q(k)$ closest to $\x^*$, there exists $\mathbf h\in [0,1]^n$ such that  $\x^*+\mathbf h \in Q(k)$ and
$\|\mathbf h\| \le {\sqrt n\over k}$. Then, by Taylor's theorem, we have that
\begin{equation}\label{eqTaylor}
f(\x^*+\mathbf h)= f(\x^*)+\mathbf h^T\nabla f(\x^*)+{1\over 2}\mathbf h^T\nabla^2 f(\zeta)\mathbf h,
\end{equation}
for some point $\zeta$ lying in the segment $[\x^*,\x^*+\mathbf h]\subseteq [0,1]^n$.
}

{
Assume first that the global minimizer $\x^*$ lies in the interior of $[0,1]^n$. Then $\nabla f(\x^*)=0$ and thus
$$f_{\min,Q(k)} -f(\x^*) \le f(\x^*+\mathbf h)- f(\x^*) \le C\|\mathbf h\|^2 \le  {n C\over k^2},$$
after setting $C:=\max_{\zeta \in [0,1]^n} \|\nabla^2 f(\zeta)\|/2$.
}

{
Assume now that $\x^*$ lies on the boundary of $[0,1]^n$ and let $I_0$ (resp., $I_1$, $I$) denote the set of indices $i\in [n]$ for which $x_i^*=0$ (resp., $x_i^*=1$, $x_i^*\in (0,1)$). Define the polynomial $g(y)=f(y,0,\ldots, 0, 1,\ldots,1)$ (with 0 at the positions $i\in I_0$ and 1 at the positions $i\in I_1$) in the variable $y\in \oR^{|I|}$. Then $\x_I^*=(x_i^*)_{i\in I}$ is a global minimizer of $g$ over $[0,1]^{|I|}$ which lies in the interior.
So we may apply the preceding reasoning to the polynomial $g$ and conclude that
$g_{\min,Q(k)}-g(\x^*_I) \le {C'\over k^2}$ for some constant $C'$ (depending on $g$ and thus on $f$).
As $f_{\min, Q(k)}\le g_{\min, Q(k)}$ and $f(\x^*)=g(\x^*_I)$ the result follows.
}
\end{proof}

%In other words, at a cost of $k^n$ function evaluations, one has \tcolblue{ for the parameter $f_{\min,Q(k)}$} an approximation error of order $O\left(\frac{1}{k}\right)$.
%By comparison, for the $f^H_k$ bound, one may obtain an approximation error of the same order
%using ${n + k \choose k}$ calculations that are comparable to function evaluations.
%For fixed values of $k$, the bounds $f^H_k$ are therefore superior to $f_{\min,Q(k)}$ in the sense of a superior error bound.
%In other words,  both bounds $f_{\min,Q(k)}$ and $f^H_k$ have the same errror estimate of $O(1/k)$ in terms of $k$ (when a rational minimizer exists).

{Therefore  the bounds $f_{\min, Q(k)}$ obtained through grid search have a faster convergence rate than the bounds $f^H_k$.
However, for any fixed value of $k$, for  the bound $f^H_k$ one needs a polynomial number $O(n^k)$ of computations (similar to function evaluations), while computing  the bound $f_{\min,Q(k)}$ requires an  exponential number $k^n$ of function evaluations. Hence the `measure-based' guided search producing the bounds $f^H_k$ is superior to  the brute force grid search technique in terms of complexity.}

\section{Obtaining feasible points $\x$ with $f(\x) \le f^H_k$}\label{sec:feasible points}

In this section we describe how to generate a point  {$\x \in \K \subseteq  [0,1]^n$}  such
that $f(\x) \le f^{H}_k$ (or such that $f(\x) \le f^{H}_k+ \epsilon$ for some small $\epsilon > 0$).

We will discuss in turn:
\begin{itemize}
\item
the convex case (and related cases), and
\item
the general case.
\end{itemize}

\subsection{The convex case (and related cases): using the Jensen inequality}\label{secconvex}
Our main tool for treating the convex case (and related cases) will be the Jensen inequality.

\begin{lemma}[Jensen inequality]
If $\mathcal{C} \subseteq \mathbb{R}^n$ is convex, $\phi:\mathcal{C} \rightarrow \mathbb{R}$ is a convex function, and $X \in \mathcal{C}$ a random variable,
then
\[
\phi(\mathbb{E}(X)) \le \mathbb{E}(\phi(X)).
\]
\end{lemma}

\begin{theorem}
\label{thm:convex case}
Assume that $\K \subseteq [0,1]^n$ is closed and convex, and  $(\eta,\beta)\in \N^{2n}_k$ is such that
\[
f^H_k = \frac{\int_{\K} f(\x)\,\x^\eta(\1-\x)^\beta\,d\x}
{\int_{\K} \x^\eta(\1-\x)^\beta\,d\x}.
\]
Let $X = (X_1,\ldots,X_n)$ be a vector of random variables with
$X_i \sim beta(\eta_i+1,\beta_i+1)$ ($i \in [n]$).

Then one has $f(\mathbb{E}(X)) \le f^H_k$ in the following cases:
\begin{enumerate}
\item
$f$ is convex;
\item
$f$ has only nonnegative coefficients;
\item
$f$ is square-free, i.e., $f(\x)=\sum_{\alpha\in \{0,1\}^n} f_\alpha \x^\alpha$.
\end{enumerate}
\end{theorem}
\begin{proof}
The proof uses the fact that, by construction,
\[
f^H_k = \mathbb{E}(f(X)).
\]
Thus the first item follows immediately from Jensen's inequality.
For the proof of the second item, recall that
\[
f^H_k = \mathbb{E}(f(X))=\sum_{\alpha\in\oN^n}f_{\alpha}\prod_{i=1}^n\mathbb{E}(X_i^{\alpha_i})
\]
where we now assume $f_\alpha \ge 0$ for all $\alpha$. Since $\phi(X_i) = X_i^{\alpha_i}$ is convex on $[0,1]$ ($i \in [n]$), Jensen's inequality yields
$\mathbb{E}(X_i^{\alpha_i}) \ge [\mathbb{E}(X_i)]^{\alpha_i}$. Thus
\[
f^H_k \ge \sum_{\alpha\in\oN^n}f_{\alpha} \mathbb{E}(X)^\alpha,
\]
as required.
For the third item, where $f$ is assumed square-free, one has
\[
f^H_k = \mathbb{E}(f(X))=\sum_{\alpha\in\oN^n}f_{\alpha}\prod_{i=1}^n\mathbb{E}(X_i^{\alpha_i})
\]
where all $\alpha \in \{0,1\}^n$ so that $\mathbb{E}(X_i^{\alpha_i}) = [\mathbb{E}(X_i)]^{\alpha_i}$, and consequently
\[
f^H_k = \sum_{\alpha\in\oN^n}f_{\alpha} \mathbb{E}(X)^\alpha.
\]
This completes the proof.
\end{proof}

\subsection{The general case}\label{secmode}

\subsubsection*{Sampling}
One may generate  random samples $\x\in \K$  from the
 density $\sigma$ on $\K$ using the well-known {\em method of conditional distributions}
 (see e.g., \cite[Section 8.5.1]{law07}).
 For $\K = [0,1]^n$, the procedure is described in detail in \cite[Section 3]{deklerk}.
 In this way one may obtain, with high probability, a point  $\x \in \K$ with $f(\x) \le f^H_k+ \epsilon$, for any given $\epsilon > 0$. (The size of the sample
 depends on $\epsilon$.) Here we only mention that this procedure may be done in time polynomial in $n$ and $1/\epsilon$; for details the reader is referred to
 \cite[Section 3]{deklerk}.

\subsubsection*{A heuristic based on the mode}
 As an alternative, one may consider the heuristic that returns the mode (i.e.,\ maximizer) of the density function $\sigma$ as a candidate solution.
 By way of illustration, recall that in
  Example \ref{ex:1} the mode was a good approximation of the global minimizer for $\sigma$ of degree $50$; see Figure \ref{figuresos}.
 The mode may be calculated one variable at a time using \eqref{eq:mode}.

 In Section \ref{sec:numerical results} below, we will illustrate the performance of all the strategies described in this section on numerical examples.

\section{Numerical examples}
\label{sec:numerical results}

In this section we will present numerical examples to illustrate the behavior  of the sequences of upper bounds, and of the techniques to obtain feasible
points.

We consider several well-known polynomial test functions from global optimization (also used in \cite{deklerk}), that are listed in Table \ref{tabletest}, where we set
\begin{equation*}
f_{\max,\K} \: :=\,\max\,\{f(\x):\: \x\in\K\,\}.
\end{equation*}

Note that the Booth and Matyas functions are convex. Note also that the functions have a rational minimizer in the hypercube (except the  Styblinski-Tang function).

\ignore{
{\tiny
\begin{center}
\begin{table}[h!]
\caption{Test functions}\label{tabletest}
\label{table:example}
  \begin{tabular}{| m{2.5cm} | m{5cm} |m{5.5cm} | m{1.2cm} |}
    \hline
    Name & Formula & Minimum ($f_{\min,\K}$) & Search domain ($\K$) \\ \hline
    Booth Function & $f=(20x_1+40x_2-37)^2+(40x_1+20x_2-35)^2$ & $f(0.55,0.65)=0$ & $[0,1]^2$ \\ \hline
    Matyas Function & $f=0.26[(20x_1-10)^2+(20x_2-10)^2]-0.48(20x_1-10)(20x_2-10)$ &$f(0.5,0.5)=0$ & $[0,1]^2$ \\ \hline
    Motzkin Polynomial & $f=(4x_1-2)^4(4x_2-2)^2+(4x_1-2)^2(4x_2-2)^4-3(4x_1-2)^2(4x_2-2)^2+1$ & $f({1\over4},{1\over4})=f({1\over4},{3\over4})=f({3\over4},{1\over4})=f({3\over4},{3\over4})=0$ & $[0,1]^2$ \\ \hline
    Three-Hump Camel Function & $f=2(10x_1-5)^2-1.05(10x_1-5)^4+{1\over 6}(10x_1-5)^6+(10x_1-5)(10x_2-5)+(10x_2-5)^2$ &$f(0.5,0.5)=0$ & $[0,1]^2$\\ \hline
    Styblinski-Tang Function  & $f=\sum_{i=1}^n{1\over 2}(10x_i-5)^4-8(10x_i-5)^2+{5\over 2}(10x_i-5)$ & $f(0.20906466,\dots,0.20906466)
    =-39.16599n$& $[0,1]^n$\\ \hline
    Rosenbrock Function & $f=\sum_{i=1}^{n-1}100(4.096x_{i+1}-2.048-(4.096x_i-2.048)^2)^2+(4.096x_i-3.048)^2 $ & $f({3048\over4096},\dots,{3048\over4096})=0$ & $[0,1]^n$ \\
    \hline
  \end{tabular}
  \end{table}
  \end{center}
}
%--------------------------------------------------------------------------------------------------------

We start by listing the upper bounds $f^H_k$ for these test functions in Table \ref{tab:overview} for densities with degree up to $k=50$.

We performed the computation using Matlab on a Laptop with Intel Core i7 4600U CPU (2.10 GHz) and 8 GB RAM. The generalized eigenvalue problem was solved in Matlab using the {\tt eig} function.

{\tiny
\begin{center}
\begin{table}[h!]
\caption{$f^H_k$ for Booth, Matyas, Motzkin, Three--Hump Camel, Styblinski--Tang and Rosenbrock Functions. \label{tab:overview}}
  \begin{tabular}{|c| c | c | c | m{1.3cm} | m{1.5cm} | m{1.4cm} |m{1.4cm}|m{1.4cm}|}\hline
  $k$&     Booth & Matyas    & Motzkin & T-H. Camel & St.-Tang ($n=2$)& Rosen. ($n=2$)& Rosen. ($n=3$)& Rosen. ($n=4$)  \\ \hline
  $1$&  $280.667$& $17.3333$ &$4.2000$ &$265.77$    &$-12.5$          &$303.16$       &$794.818$      &$1289.9$\\ \hline
  $2$& $250.667$ & $12.0000$ &$2.1886$ &$86.091$    &$-17.381$        &$235.68$       &$603.931$      &$1097.7$\\ \hline
  $3$&  $ 214.0$ & $11.0667$ &$2.1886$ &$86.091$    &$-21.548$        &$177.91$       &$536.449$      &$906.76$\\ \hline
  $4$&   $184.0$ & $8.8000$  &$1.2743$ &$40.593$    &$-26.429$        &$148.6$        &$478.673$      &$839.28$\\ \hline
  $5$&   $172.0$ & $8.1333$  &$1.2743$ &$40.593$    &$-28.929$        &$142.2$        &$411.191$      &$781.51$\\ \hline
  $6$& $151.333$ & $6.9867$  &$1.0218$ &$24.354$    &$-31.429$        &$130.43$       &$343.863$      &$714.02$\\ \hline
  $7$& $143.905$ & $6.5524$  &$1.0218$ &$24.354$    &$-32.778$        &$120.17$       &$314.559$      &$646.68$\\ \hline
  $8$& $130.762$ & $5.9048$  &$0.8912$ &$17.322$    &$-34.127$        &$103.43$       &$296.24$       &$579.2$\\ \hline
  $9$& $125.429$ & $5.6190$  &$0.8912$ &$17.322$    &$-34.921$        &$100.03$       &$266.936$      &$511.86$\\ \hline
 $10$& $117.571$ & $5.2245$  &$0.8538$ &$13.867$    &$-35.714$        &$91.011$       &$252.003$      &$482.56$\\ \hline
 $11$& $109.556$ & $5.0317$  &$0.8538$ &$13.867$    &$-36.956$        &$87.425$       &$239.06$       &$460.14$\\ \hline
 $12$& $106.222$ & $4.7778$  &$0.8384$ &$10.534$    &$-38.305$        &$76.959$       &$225.146$      &$430.83$\\ \hline
 $13$& $99.4545$ & $4.6444$  &$0.8384$ &$10.534$    &$-39.516$        &$75.033$       &$212.057$      &$406.9$\\ \hline
 $14$& $94.7407$ & $4.4741$  &$0.8366$ &$8.6752$    &$-40.31$         &$69.148$       &$203.723$      &$377.6$\\ \hline
 $15$& $90.6667$ & $4.3798$  &$0.8339$ &$8.6752$    &$-41.003$        &$66.266$       &$189.252$      &$362.66$\\ \hline
 $16$& $85.6364$ & $4.2618$  &$0.8336$ &$7.2466$    &$-42.483$        &$60.434$       &$179.188$      &$349.718$\\ \hline
 $17$& $83.0909$ & $4.1939$  &$0.8242$ &$7.2466$    &$-43.694$        &$59.243$       &$169.714$      &$334.462$\\ \hline
 $18$& $78.6434$ & $4.1102$  &$0.8139$ &$6.1763$    &$-44.905$        &$55.276$       &$163.392$      &$321.52$\\ \hline
 $19$& $75.8648$ & $4.0606$  &$0.8062$ &$6.1763$    &$-45.598$        &$52.947$       &$155.662$      &$309.927$\\ \hline
 $20$& $73.5152$ & $4.0000$  &$0.8025$ &$5.3826$    &$-46.291$        &$49.381$       &$150.066$      &$294.517$\\ \hline
 $25$& $61.6535$ & $3.4324$  &$0.7762$ &$4.2267$    &$-49.633$        &$40.704$       &$121.272$      &$242.747$\\ \hline
 $30$& $53.1228$ & $2.8927$  &$0.7474$ &$3.1892$    &$-52.976$        &$33.338$       &$101.914$      &$205.889$\\ \hline
 $35$& $46.5982$ & $2.5989$  &$0.7067$ &$2.7367$    &$-55.193$        &$28.72$        &$86.9293$      &$177.821$\\ \hline
 $40$& $41.6416$ & $2.2609$  &$0.6625$ &$2.2626$    &$-57.411$        &$24.883$       &$75.5008$      &$155.681$\\ \hline
 $45$& $37.4988$ & $2.0800$  &$0.6254$ &$2.0337$    &$-58.998$        &$21.984$       &$67.1078$      &$138.990$\\ \hline
 $50$& $34.0573$ & $1.8595$  &$0.5914$ &$1.7768$    &$-60.536$        &$19.739$       &$59.6395$      &$124.115$\\ \hline
  \end{tabular}
  \end{table}
\end{center}
}
One notices that the observed convergence rate is more-or-less in line with the $O(1/k)$ bound.

In a next experiment, we compare the Handelman-type densities ($f^H_k$ bounds) to SOS densities ($f^{sos}_{k/2}$ bounds);
see Tables \ref{tab:compare1} and \ref{tab:compare2}.

{\tiny
\begin{center}
\begin{table}[h!]
\caption{Comparison of  two upper bounds for Booth, Matyas, Three--Hump Camel and Motzkin Functions \label{tab:compare1}}
\label{table:result1}
  \begin{tabular}{| c | c | c | c | c | c | c | c | c |}
    \hline
\multirow{2}{*}{degree $k$} & \multicolumn{2}{c|}{Booth} & \multicolumn{2}{c|}{Matyas} & \multicolumn{2}{m{3cm}|}{Three--Hump Camel}& \multicolumn{2}{c|}{Motzkin} \\ \cline{2-9}
     & $f^{sos}_{k/2}$       & $f^H_k$ & $f^{sos}_{k/2}$       & $f^H_k$  & $f^{sos}_{k/2}$   & $f^H_k$  & $f^{sos}_{k/2}$ &$f^H_k$\\ \hline
$1$  &     &       $280.667$ & &         $17.3333$   &&         $265.77$&&               $4.2$\\ \hline
$2$  & $244.680$ & $250.667$ & $8.26667$ & $12.0$ & $265.774$ & $86.091$  & $4.2$      & $2.1886$\\ \hline
$3$  &&            $214.0$&&               $11.0667$&&          $86.091$&&               $2.1886$\\ \hline
$4$  & $162.486$ & $184.0$ & $5.32223$ & $8.8000$ & $29.0005$ & $40.593$  & $1.06147$  & $1.2743$\\ \hline
$5$ &&             $172.0$&&             $8.1333$&&             $40.593$&&                $1.2743$\\ \hline
$6$  & $118.383$ & $151.333$ & $4.28172$ & $6.9867$ & $29.0005$ & $24.354$  & $1.06147$  & $1.0218$\\ \hline
$7$ &&              $143.905$&&            $6.5524$&&              $24.354$&&              $1.0218$\\ \hline
$8$  & $97.6473$ & $130.762$ & $3.89427$ & $5.9048$ & $9.58064$ & $17.322$  & $0.829415$ & $0.8912$\\ \hline
$9$ &&              $125.429$&&            $5.6190$ &&            $17.322$&&               $0.8912$\\ \hline
$10$  & $69.8174$ & $117.571$ & $3.68942$ & $5.2245$ & $9.58064$ & $13.867$  & $0.801069$ & $0.8538$\\ \hline
$11$ &&             $109.556$&&             $5.0317$&&             $13.867$&&               $0.8538$\\ \hline
$12$  & $63.5454$ & $106.222$ & $2.99563$ & $4.7778$ & $4.43983$ & $10.534$  & $0.801069$ & $0.8384$\\ \hline
$13$ &&              $99.4545$&&            $4.6444$ &&            $10.534$&&               $0.8384$\\ \hline
$14$  & $47.0467$ & $94.7407$ & $2.54698$ & $4.4741$ & $4.43983$ & $8.6752$  & $0.708889$ & $0.8366$\\ \hline
$15$ &&             $90.6667$&&             $4.3798$&&             $8.6752$&&               $0.8339$\\ \hline
$16$  & $41.6727$ & $85.6364$ & $2.04307$ & $4.2618$ & $2.55032$ & $7.2466$  & $0.565553$ & $0.8336$\\ \hline
$17$ &&             $83.0909$&&             $4.1939$&&             $7.2466$&&               $0.8242$\\ \hline
$18$  & $34.2140$ & $78.6434$ & $1.83356$ & $4.1102$ & $2.55032$ & $6.1763$  & $0.565553$ & $0.8139$\\ \hline
$19$ &&             $75.8648$&&             $4.0606$&&             $6.1763$&&               $0.8062$\\ \hline
$20$ & $28.7248$ & $73.5152$ & $1.47840$ & $4.0000$ & $1.71275$ & $5.3826$  & $0.507829$ & $0.8025$\\
\hline
  \end{tabular}
  \end{table}
\end{center}
}
%--------------------------------------------------------------------------------------------------------
{\tiny
\begin{center}
\begin{table}[h!]
\caption{Comparison of two upper bounds for Styblinski--Tang and Rosenbrock Functions \label{tab:compare2}}
\label{table:result2}
  \begin{tabular}{| c | c | c | c | c | c | c| c | c|}
\hline
\multirow{2}{*}{degree $k$}&\multicolumn{2}{c|}{Sty.--Tang ($n=2$)}&\multicolumn{2}{c|}{Rosenb. ($n=2$)}&\multicolumn{2}{c|}{Rosenb. ($n=3$)}&\multicolumn{2}{c|}{Rosenb. ($n=4$)}\\ \cline{2-9}
          & $f^{sos}_{k/2}$       & $f^H_k$           & $f^{sos}_{k/2}$        & $f^H_k$       & $f^{sos}_{k/2}$       & $f^H_k$        & $f^{sos}_{k/2}$       & $f^H_k$\\ \hline
    $1$&&                $-12.5$&&                     $303.16$    &           & $794.818$    &           & $1289.9$\\ \hline
    $2$   & $-12.9249$ & $-17.381$      & $214.648$  & $235.68$    & $629.086$ & $603.931$    & $1048.19$ & $1097.7$     \\ \hline
    $3$&&                $-21.548$&&                   $177.91$    &           & $536.449$    &           & $906.76$\\ \hline
    $4$   & $-25.7727$ & $-26.429$      & $152.310$  & $148.6$     & $394.187$ & $478.673$    & $690.332$ & $839.28$   \\ \hline
    $5$&&                $-28.929$&&                   $142.2$     &           & $411.191$    &           & $781.51$\\ \hline
    $6$   & $-34.4030$ & $-31.429$      & $104.889$  & $130.43$    & $295.811$ & $343.863$    & $536.367$ & $714.02$  \\ \hline
    $7$&&                $-32.778$&&                   $120.17$    &           & $314.559$    &           & $646.68$\\ \hline
    $8$   & $-41.4436$ & $-34.127$      & $75.6010$  & $103.43$    & $206.903$ & $296.24$     & $382.729$ & $579.2$    \\ \hline
    $9$&&                $-34.921$&&                   $100.03$    &           & $266.936$    &           & $511.86$\\ \hline
    $10$   & $-45.1032$ & $-35.714$      & $51.5037$  &$91.011$    & $168.135$ & $252.003$    & $314.758$ & $482.56$    \\ \hline
    $11$&&                $-36.956$&&                  $87.425$    &           & $239.06$     &           & $460.14$\\ \hline
    $12$   & $-51.0509$ & $-38.305$      & $41.7878$  &$76.959$    & $121.558$ & $225.146$    & $236.709$ & $430.83$    \\ \hline
    $13$&&                $-39.516$&&                  $75.033$    &           & $212.057$    &           & $406.9$\\ \hline
    $14$   & $-56.4050$ & $-40.31$      & $30.1392$  & $69.148$    & $101.953$ & $203.723$    & $202.674$ & $377.6$    \\ \hline
    $15$&&                $-41.003$&&                  $66.266$    &           & $189.252$    &           & $362.66$\\ \hline
    $16$   & $-58.6004$ & $-42.483$      & $25.8329$  & $60.434$   & $77.4797$ & $179.188$    & $156.295$ & $349.718$    \\ \hline
    $17$&&                $-43.694$&&                  $59.243$    &           & $169.714$    &           & $334.462$   \\ \hline
    $18$   & $-60.7908$ & $-44.905$      & $19.4972$  & $55.276$   & $66.6954$ & $163.392$    & $137.015$ & $321.52$   \\
    \hline
  \end{tabular}
\end{table}
\end{center}
}
}

{\tiny
\begin{center}
\begin{table}[h!]
\caption{Test functions}\label{tabletest}
\label{table:example}
  \begin{tabular}{| m{1.5cm} | m{3.2cm} |m{2.7cm} |m{2.7cm} | m{1.2cm} |}
    \hline
    Name & Formula & Minimum ($f_{\min,\K}$) & Maximum ($f_{\max,\K}$) & Search domain ($\K$) \\ \hline
    Booth Function & $f=(20x_1+40x_2-37)^2+(40x_1+20x_2-35)^2$ & $f(0.55,0.65)=0$ & $f(0,0)=2594$ & $[0,1]^2$ \\ \hline
    Matyas Function & $f=0.26[(20x_1-10)^2+(20x_2-10)^2]-0.48(20x_1-10)(20x_2-10)$ &$f(0.5,0.5)=0$ & $f(0,1)=100$ & $[0,1]^2$ \\ \hline
    Motzkin Polynomial & $f=(4x_1-2)^4(4x_2-2)^2+(4x_1-2)^2(4x_2-2)^4-3(4x_1-2)^2(4x_2-2)^2+1$ & $f({1\over4},{1\over4})=f({1\over4},{3\over4})=f({3\over4},{1\over4})=f({3\over4},{3\over4})=0$ & $f(1,1)=81$ & $[0,1]^2$ \\ \hline
    Three-Hump Camel Function & $f=2(10x_1-5)^2-1.05(10x_1-5)^4+{1\over 6}(10x_1-5)^6+(10x_1-5)(10x_2-5)+(10x_2-5)^2$ &$f(0.5,0.5)=0$ &$ f(1,1)=2047.92$ & $[0,1]^2$\\ \hline
    Styblinski-Tang Function  & $f=\sum_{i=1}^n{1\over 2}(10x_i-5)^4-8(10x_i-5)^2+{5\over 2}(10x_i-5)$ & $f(0.209,\dots,0.209)
    =-39.16599n$& $f(1,\dots,1)=125n$ & $[0,1]^n$\\ \hline
    Rosenbrock Function & $f=\sum_{i=1}^{n-1}100(4.096x_{i+1}-2.048-(4.096x_i-2.048)^2)^2+(4.096x_i-3.048)^2 $ & $f({3048\over4096},\dots,{3048\over4096})=0$ & $f(0,\dots,0)=3905.93(n-1)$ & $[0,1]^n$ \\
    \hline
  \end{tabular}
  \end{table}
  \end{center}
}
%--------------------------------------------------------------------------------------------------------
{We start by listing the relative gaps $RG(\%)= \frac{f^H_k -f_{\min,\K}}{f_{\max,\K} -f_{\min,\K}} \times 100$
for these test functions in Table \ref{tab:overview} for densities with degree up to $k=50$.}

{\tiny
\begin{center}
\begin{table}[h!]
\caption{Relative gaps of $f^H_k$ for test functions in Table \ref{tabletest}. \label{tab:overview}}
  \begin{tabular}{|c| c | c | c | c | c | c |c|c|}\hline
  $k$&     Booth & Matyas    & Motzkin & T-H. Camel & St.-Tang ($n=2$)& Rosen. ($n=2$)& Rosen. ($n=3$)& Rosen. ($n=4$)  \\ \hline
  $1$&$10.8199$&$17.3333$&$5.1852$&$12.9776$&$20.0499$&$7.7615$&$10.1745$&$11.0081$\\ \hline
$2$&$9.6633$&$12.0000$&$2.7020$&$4.2038$&$18.5633$&$6.0339$&$7.7310$&$9.3678$\\ \hline
$3$&$8.2498$&$11.0667$&$2.7020$&$4.2038$&$17.2942$&$4.5549$&$6.8671$&$7.7383$\\ \hline
$4$&$7.0933$&$8.8000$&$1.5732$&$1.9822$&$15.8076$&$3.8045$&$6.1275$&$7.1624$\\ \hline
$5$&$6.6307$&$8.1333$&$1.5732$&$1.9822$&$15.0461$&$3.6406$&$5.2637$&$6.6694$\\ \hline
$6$&$5.8340$&$6.9867$&$1.2615$&$1.1892$&$14.2847$&$3.3393$&$4.4018$&$6.0935$\\ \hline
$7$&$5.5476$&$6.5524$&$1.2615$&$1.1892$&$13.8738$&$3.0766$&$4.0267$&$5.5188$\\ \hline
$8$&$5.0409$&$5.9048$&$1.1002$&$0.8458$&$13.4630$&$2.6480$&$3.7922$&$4.9429$\\ \hline
$9$&$4.8354$&$5.6190$&$1.1002$&$0.8458$&$13.2211$&$2.5610$&$3.4171$&$4.3682$\\ \hline
$10$&$4.5324$&$5.2245$&$1.0541$&$0.6771$&$12.9796$&$2.3301$&$3.2259$&$4.1182$\\ \hline
$11$&$4.2234$&$5.0317$&$1.0541$&$0.6771$&$12.6013$&$2.2383$&$3.0602$&$3.9269$\\ \hline
$12$&$4.0949$&$4.7778$&$1.0351$&$0.5144$&$12.1905$&$1.9703$&$2.8821$&$3.6767$\\ \hline
$13$&$3.8340$&$4.6444$&$1.0351$&$0.5144$&$11.8216$&$1.9210$&$2.7146$&$3.4725$\\ \hline
$14$&$3.6523$&$4.4741$&$1.0328$&$0.4236$&$11.5798$&$1.7703$&$2.6079$&$3.2225$\\ \hline
$15$&$3.4952$&$4.3798$&$1.0295$&$0.4236$&$11.3687$&$1.6965$&$2.4226$&$3.0950$\\ \hline
$16$&$3.3013$&$4.2618$&$1.0291$&$0.3539$&$10.9180$&$1.5472$&$2.2938$&$2.9845$\\ \hline
$17$&$3.2032$&$4.1939$&$1.0175$&$0.3539$&$10.5491$&$1.5167$&$2.1725$&$2.8543$\\ \hline
$18$&$3.0317$&$4.1102$&$1.0048$&$0.3016$&$10.1803$&$1.4152$&$2.0916$&$2.7439$\\ \hline
$19$&$2.9246$&$4.0606$&$0.9953$&$0.3016$&$9.9692$&$1.3556$&$1.9926$&$2.6449$\\ \hline
$20$&$2.8340$&$4.0000$&$0.9907$&$0.2628$&$9.7582$&$1.2643$&$1.9210$&$2.5134$\\ \hline
$25$&$2.3768$&$3.4324$&$0.9583$&$0.2064$&$8.7403$&$1.0421$&$1.5524$&$2.0716$\\ \hline
$30$&$2.0479$&$2.8927$&$0.9227$&$0.1557$&$7.7221$&$0.8535$&$1.3046$&$1.7571$\\ \hline
$35$&$1.7964$&$2.5989$&$0.8725$&$0.1336$&$7.0469$&$0.7353$&$1.1128$&$1.5175$\\ \hline
$40$&$1.6053$&$2.2609$&$0.8179$&$0.1105$&$6.3713$&$0.6371$&$0.9665$&$1.3286$\\ \hline
$45$&$1.4456$&$2.0800$&$0.7721$&$0.0993$&$5.8880$&$0.5628$&$0.8591$&$1.1861$\\ \hline
$50$&$1.3129$&$1.8595$&$0.7301$&$0.0868$&$5.4195$&$0.5054$&$0.7634$&$1.0592$\\ \hline
  \end{tabular}
  \end{table}
\end{center}
}
One notices that the observed convergence rate is more-or-less in line with the $O(1/k)$ bound.

In a next experiment, we compare the Handelman-type densities ($RG(\%)$ by $f^H_k$ bounds) to SOS densities
 (we still use the notation $RG(\%)=(f^{sos}_k-f_{\min,\K})/(f_{\max,\K}-f_{\min,\K}) \times 100$); we also compare their computation times (in seconds),
for which we use the approaches described in Section \ref{sec:computing procedure},
  and we assume that the values $\gamma_{(\eta,\beta)}$ for all $(\eta,\beta)\in \oN_{k+d}^{2n}$ and the moments of the Lebesgue measure on $\K=[0,1]^n$ are computed beforehand;
see Tables \ref{tab:compare1}, \ref{tab:compare2} and \ref{tab:comparenew1}. We performed the computation using Matlab on a Laptop with Intel Core i7-4600U CPU (2.10 GHz) and 8 GB RAM. The generalized eigenvalue computation was done in Matlab using the eig function.

{\tiny
\begin{center}
\begin{table}[h!]
\caption{Comparison of  two upper bounds for Booth, Matyas and Three--Hump Camel functions in relative gaps and computation times (sec.) \label{tab:compare1}}
  \begin{tabular}{| m{0.3cm} | m{0.7cm} | m{0.7cm} | m{0.7cm} | m{0.7cm} |m{0.7cm} | m{0.7cm} |m{0.7cm} | m{0.7cm} |m{0.7cm} | m{0.7cm} |m{0.7cm} | m{0.7cm} |}
    \hline
\multirow{3}{*}{$k$} & \multicolumn{4}{c|}{Booth} & \multicolumn{4}{c|}{Matyas} & \multicolumn{4}{c|}{Three--Hump Camel} \\ \cline{2-13}
     & \multicolumn{2}{c|}{$f^{sos}_{k/2}$}       & \multicolumn{2}{c|}{$f^H_k$} & \multicolumn{2}{c|}{$f^{sos}_{k/2}$}       & \multicolumn{2}{c|}{$f^H_k$}  & \multicolumn{2}{c|}{$f^{sos}_{k/2}$}   & \multicolumn{2}{c|}{$f^H_k$}  \\ \cline{2-13}
     & $RG(\%)$ & time & $RG(\%)$ & time &$RG(\%)$ & time &$RG(\%)$ & time &$RG(\%)$ & time &$RG(\%)$ & time \\ \hline
 $2$&$9.433$&$0.0007$&$9.663$&$0.0001$&$8.267$&$0.0009$&$ 12.0$&$0.0001$&$  12.98$&$0.0008$&$ 4.204$&$0.0001$\\ \hline
$4$&$6.264$&$0.0006$&$7.093$&$0.0003$&$5.322$&$0.0005$&$  8.8$&$0.0003$&$  1.416$&$0.0006$&$ 1.982$&$0.0002$\\ \hline
$6$&$4.564$&$0.0008$&$5.834$&$0.0008$&$4.282$&$0.0009$&$6.987$&$0.0007$&$  1.416$&$0.0011$&$ 1.189$&$0.0007$\\ \hline
$8$&$3.764$&$0.0015$&$5.041$&$0.0025$&$3.894$&$0.0017$&$5.905$&$0.0018$&$ 0.4678$&$ 0.002$&$0.8458$&$0.0017$\\ \hline
$10$&$2.691$&$0.0025$&$4.532$&$0.0038$&$3.689$&$0.0033$&$5.224$&$0.0039$&$ 0.4678$&$0.0035$&$0.6771$&$0.0037$\\ \hline
$12$&$ 2.45$&$0.0047$&$4.095$&$0.0065$&$2.996$&$0.0056$&$4.778$&$0.0074$&$ 0.2168$&$0.0086$&$0.5144$&$0.0063$\\ \hline
$14$&$1.814$&$0.0072$&$3.652$&$0.0109$&$2.547$&$0.0102$&$4.474$&$0.0112$&$ 0.2168$&$0.0128$&$0.4236$&$0.0117$\\ \hline
$16$&$1.607$&$0.0097$&$3.301$&$0.0177$&$2.043$&$0.0131$&$4.262$&$0.0178$&$ 0.1245$&$0.0139$&$0.3539$&$0.0179$\\ \hline
$18$&$1.319$&$0.0146$&$3.032$&$0.0276$&$1.834$&$0.0226$&$ 4.11$&$0.0266$&$ 0.1245$&$0.0377$&$0.3016$&$0.027$\\ \hline
$20$&$1.107$&$0.0242$&$2.834$&$0.0391$&$1.478$&$0.0329$&$4.0$&$0.0384$&$0.08363$&$0.0312$&$0.2628$&$0.0397$\\ \hline
  \end{tabular}
  \end{table}
\end{center}
}

{\tiny
\begin{center}
\begin{table}[h!]
\caption{Comparison of  two upper bounds for Motzkin, Styblinski-Tang ($n=2$) and Rosenbrock ($n=2$) functions in relative gaps and computation times (sec.) \label{tab:compare2}}
  \begin{tabular}{| m{0.3cm} | m{0.7cm} | m{0.7cm} | m{0.7cm} | m{0.7cm} |m{0.7cm} | m{0.7cm} |m{0.7cm} | m{0.7cm} |m{0.7cm} | m{0.7cm} |m{0.7cm} | m{0.7cm} |}
    \hline
\multirow{3}{*}{$k$} & \multicolumn{4}{c|}{Motzkin} & \multicolumn{4}{c|}{Sty.--Tang ($n=2$)} & \multicolumn{4}{c|}{Rosenb. ($n=2$)} \\ \cline{2-13}
     & \multicolumn{2}{c|}{$f^{sos}_{k/2}$}       & \multicolumn{2}{c|}{$f^H_k$} & \multicolumn{2}{c|}{$f^{sos}_{k/2}$}       & \multicolumn{2}{c|}{$f^H_k$}  & \multicolumn{2}{c|}{$f^{sos}_{k/2}$}   & \multicolumn{2}{c|}{$f^H_k$}  \\ \cline{2-13}
     & $RG(\%)$ & time & $RG(\%)$ & time &$RG(\%)$ & time &$RG(\%)$ & time &$RG(\%)$ & time &$RG(\%)$ & time \\ \hline
 $2$&$  5.185$&$0.0008$&$ 2.702$&$0.0001$&$19.92$&$0.0008$&$18.56$&$0.0001$&$ 5.495$&$ 0.001$&$6.034$&$0.0001$\\ \hline
$4$&$   1.31$&$0.0005$&$ 1.573$&$0.0003$&$16.01$&$0.0005$&$15.81$&$0.0002$&$ 3.899$&$0.0009$&$3.804$&$0.0003$\\ \hline
$6$&$   1.31$&$0.0009$&$ 1.261$&$0.0009$&$13.38$&$0.0009$&$14.28$&$0.0008$&$ 2.685$&$0.0018$&$3.339$&$0.0013$\\ \hline
$8$&$  1.024$&$0.0016$&$   1.1$&$ 0.002$&$11.23$&$0.0016$&$13.46$&$0.0021$&$ 1.936$&$0.0031$&$2.648$&$0.0034$\\ \hline
$10$&$  0.989$&$0.0034$&$ 1.054$&$0.0043$&$10.12$&$0.0028$&$12.98$&$0.0037$&$ 1.319$&$0.0031$&$ 2.33$&$0.0057$\\ \hline
$12$&$  0.989$&$0.0062$&$ 1.035$&$ 0.006$&$8.308$&$0.0063$&$12.19$&$0.0078$&$  1.07$&$0.0049$&$ 1.97$&$ 0.008$\\ \hline
$14$&$ 0.8752$&$0.0096$&$ 1.033$&$0.0168$&$6.678$&$0.0097$&$11.58$&$0.0177$&$0.7716$&$0.0083$&$ 1.77$&$ 0.012$\\ \hline
$16$&$ 0.6982$&$0.0216$&$ 1.029$&$0.0179$&$6.009$&$ 0.014$&$10.92$&$0.0214$&$0.6614$&$0.0119$&$1.547$&$0.0237$\\ \hline
$18$&$ 0.6982$&$0.0242$&$ 1.005$&$0.0266$&$5.342$&$0.0231$&$10.18$&$0.0358$&$0.4992$&$0.0198$&$1.415$&$0.0264$\\ \hline
$20$&$ 0.6269$&$0.0298$&$0.9907$&$ 0.046$&$ 4.36$&$0.0286$&$9.758$&$ 0.042$&$0.4455$&$0.0324$&$1.264$&$0.0383$\\ \hline
  \end{tabular}
  \end{table}
\end{center}
}

%--------------------------------------------------------------------------------------------------------
{\tiny
\begin{center}
\begin{table}[h!]
\caption{Comparison of two upper bounds for Rosenbrock functions ($n=3,4$) in relative gaps and computation times (sec.)\label{tab:comparenew1}}
  \begin{tabular}{| m{0.3cm} | m{0.7cm} | m{0.7cm} | m{0.7cm} | m{0.7cm} |m{0.7cm} | m{0.7cm} |m{0.7cm} | m{0.7cm} |}
    \hline
\multirow{3}{*}{$k$} & \multicolumn{4}{c|}{Rosenb. ($n=3$)} & \multicolumn{4}{c|}{Rosenb. ($n=4$)} \\ \cline{2-9}
     & \multicolumn{2}{c|}{$f^{sos}_{k/2}$}       & \multicolumn{2}{c|}{$f^H_k$} & \multicolumn{2}{c|}{$f^{sos}_{k/2}$}       & \multicolumn{2}{c|}{$f^H_k$}  \\ \cline{2-9}
     & $RG(\%)$ & time & $RG(\%)$ & time &$RG(\%)$ & time &$RG(\%)$ & time  \\ \hline
 $2$&$  8.053$&$0.0033$&$7.731$&$0.0001$&$8.945$&$0.0204$&$9.368$&$0.0002$\\ \hline
$4$&$  5.046$&$0.0009$&$6.128$&$0.0007$&$5.891$&$0.0243$&$7.162$&$0.0017$\\ \hline
$6$&$  3.787$&$0.0024$&$4.402$&$0.0021$&$4.577$&$0.0111$&$6.093$&$0.0062$\\ \hline
$8$&$  2.649$&$0.0078$&$3.792$&$0.0054$&$3.266$&$0.0442$&$4.943$&$0.0228$\\ \hline
$10$&$  2.152$&$ 0.016$&$3.226$&$0.0135$&$2.686$&$0.2087$&$4.118$&$0.0699$\\ \hline
$12$&$  1.556$&$0.0355$&$2.882$&$0.0244$&$ 2.02$&$0.3774$&$3.677$&$0.1837$\\ \hline
$14$&$  1.305$&$0.0811$&$2.608$&$ 0.041$&$ 1.73$&$0.9121$&$3.222$&$ 0.431$\\ \hline
$16$&$ 0.9918$&$0.1324$&$2.294$&$0.0684$&$1.334$&$ 1.986$&$2.985$&$ 1.099$\\ \hline
$18$&$ 0.8538$&$0.2272$&$2.092$&$0.1139$&$1.169$&$ 4.279$&$2.744$&$  1.92$\\ \hline
  \end{tabular}
\end{table}
\end{center}
}

As described in Example \ref{ex:1}, there is no ordering possible in general between $f^{sos}_{k/2}$ and $f^H_k$, but one
observes that $f^{sos}_{k/2} \le f^H_k$ holds in most cases, i.e.,\ the SOS densities usually give better bounds for a given degree.
One should bear in mind though, that the $f^{sos}_{k/2}$ are in general much more expensive to compute than $f^H_k$, as discussed in Section \ref{sec:computing procedure}.
This is not really visible in the computational times presented here, since the values of $n$ in the examples are too small.

Next we consider the strategies for generating feasible points corresponding to the bounds $f^H_k$, as described in Section \ref{sec:feasible points};
see Table \ref{tab:feas1}.

%---------------------------------------------------------------------------------
{\tiny
\begin{center}
\begin{table}[h!]
\caption{\label{tab:feas1}Comparing strategies for generating feasible points for Booth, Matyas, Motzkin, and Three--Hump Camel functions.
Here, $\hat{\x}$ denotes the mode of the optimal density.}
  \begin{tabular}{|c| c |c|c| c | c | c | c | c |c|c|} \hline
\multirow{2}{*}{$k$} & \multicolumn{3}{c|}{Booth} & \multicolumn{3}{c|}{Matyas} & \multicolumn{2}{c|}{Motzkin}& \multicolumn{2}{c|}{Three-H. Camel} \\ \cline{2-11}
        & $f^H_k$&$f(\hat{\x})$&$f(\mathbb{E}(X))$&$f^H_k$ &$f(\hat{\x})$&$f(\mathbb{E}(X))$&$f^H_k$ & $f(\hat{\x})$ & $f^H_k$       & $f(\hat{\x})$       \\ \hline
%  $1$&  $280.667$& ---       &$2.8889$ &$17.3333$& ---   &$0$   &$4.2000$& --- &$265.77$& ---    \\ \hline
%  $2$& $250.667$ & ---       &$9.0$    &$12.0000$& $4.0$ &$0.4444$   &$2.1886$& --- &$86.091$& ---    \\ \hline
%  $3$&  $ 214.0$ & $194.0$   &$2.8889$ &$11.0667$& $4.0$ &$1.3889$   &$2.1886$& --- &$86.091$& ---    \\ \hline
%  $4$&   $184.0$ & $194.0$   &$9.0$    &$8.8000$&  $4.0$ &$1.0$   &$1.2743$& $1.0$ &$40.593$& ---    \\ \hline
  $5$&   $172.0$ & $96.222$  &$17.0$   &$8.1333$&  $4.0$ &$1.460$   &$1.2743$& $1.0$ &$40.593$& ---    \\ \hline
 % $6$& $151.333$ & $96.222$  &$18.0$   &$6.9867$&  $4.0$ &$1.440$   &$1.0218$& $1.0$ &$24.354$& ---    \\ \hline
%  $7$& $143.905$ & $96.222$  &$24.222$ &$6.5524$&  $4.0$ &$1.7156$   &$1.0218$& $1.0$ &$24.354$& ---    \\ \hline
%  $8$& $130.762$ & $122.0$   &$16.204$ &$5.9048$&  $4.0$ &$1.7778$   &$0.8912$& $1.0$ &$17.322$& ---    \\ \hline
%  $9$& $125.429$ & $26.0$    &$2.9796$ &$5.6190$&  $4.0$ &$1.9637$   &$0.8912$& $1.0$ &$17.322$& $25.0$    \\ \hline
 $10$& $117.571$ & $96.222$  &$25.806$ &$5.2245$&  $4.0$ &$2.0408$   &$0.8538$& $1.0$ &$13.867$& ---    \\ \hline
% $11$& $109.556$ & $26.0$    &$2.9796$ &$5.0317$&  $4.0$ &$2.1760$   &$0.8538$& $1.0$ &$13.867$& $25.0$    \\ \hline
% $12$& $106.222$ & $42.889$  &$9.0$    &$4.7778$&  $4.0$ &$2.2500$   &$0.8384$& $1.0$ &$10.534$& $0$    \\ \hline
% $13$& $99.4545$ & $26.0$    &$2.9796$ &$4.6444$&  $4.0$ &$2.3534$  &$0.8384$& $1.0$ &$10.534$& $0$    \\ \hline
% $14$& $94.7407$ & $13.592$  &$0.91358$&$4.4741$&  $4.0$ &$2.4198$   &$0.8366$& $1.0$ &$8.6752$& $0$    \\ \hline
 $15$& $90.6667$ & $27.580$  &$7.6777$ &$4.3798$&  $4.0$ &$2.5017$   &$0.8339$& $1.0$ &$8.6752$& $0.273$    \\ \hline
 %$16$& $85.6364$ & $9.0$     &$2.0$    &$4.2618$&  $4.0$ &$2.5600$   &$0.8336$& $1.0$ &$7.2466$& $0$    \\ \hline
% $17$& $83.0909$ & $17.210$  &$4.5785$ &$4.1939$&  $4.0$ &$2.6268$   &$0.8242$& $1.0$ &$7.2466$& $0$    \\ \hline
% $18$& $78.6434$ & $9.0$     &$2.0$    &$4.1102$&  $4.0$ &$2.6777$   &$0.8139$& $1.0$ &$6.1763$& $0$    \\ \hline
% $19$& $75.8648$ & $5.951$   &$0.35445$&$4.0606$&  $4.0$ &$2.7332$   &$0.8062$& $1.0$ &$6.1763$& $0.209$    \\ \hline
 $20$& $73.5152$ & $9.0$     &$2.0$    &$4.0000$&  $0.16$&$0.1111$   &$0.8025$& $1.0$ &$5.3826$& $0$    \\ \hline
 $25$& $61.6535$ & $4.5785$  &$1.8107$ &$3.4324$&  $0.3161$&$0.2404$ &$0.7762$& $1.0$    &$4.2267$& $0.1653$    \\ \hline
 $30$& $53.1228$ & $1.6403$  &$0.41428$&$2.8927$&  $0.0178$&$0.0138$ &$0.7474$& $1.0$    &$3.1892$& $0$    \\ \hline
 $35$& $46.5982$ & $1.0923$  &$0.53061$&$2.5989$&  $0.1071$&$0.0897$ &$0.7067$& $0.4214$ &$2.7367$& $0.110$    \\ \hline
 $40$& $41.6416$ & $0.8454$  &$0.64566$&$2.2609$&  $0$     &$0$ &$0.6625$& $0.2955$ &$2.2626$& $0$    \\ \hline
 $45$& $37.4988$ & $2.0$     &$0.80157$&$2.0800$&  $0$     &$0$ &$0.6254$& $0.1985$ &$2.0337$& $0.0783$    \\ \hline
 $50$& $34.0573$ & $0.9784$  &$0.22222$&$1.8595$&  $0$     &$0$ &$0.5914$& $0.1297$ &$1.7768$& $0$    \\ \hline
  \end{tabular}
  \end{table}
\end{center}
}
%----------------------------------------------------------------------------------------------------------
In Table \ref{tab:feas1}, the columns marked $f(\mathbb{E}(X))$ refer to the convex case in Theorem \ref{thm:convex case}.
The columns marked $f(\hat{\x})$ correspond to the mode $\hat{\x}$ of the optimal density; an entry `---' in these columns means that
the mode of the optimal density was not unique.

For the convex Booth and Matyas functions $f(\mathbb{E}(X))$ gives the best upper bound.
For sufficiently large $k$ the mode $\hat{\x}$ gives a better bound than $f^H_k$, indicating that this heuristic is useful in the non-convex case.

As a final comparison, we also look at the general sampling technique via the method of conditional distributions; see Tables \ref{tab:sampling1} and
\ref{tab:sampling2}. We present results for the Motzkin polynomial and the Three hump camel function.

{\tiny
\begin{center}
\begin{table}[h!]
\caption{\label{tab:sampling1}Sampling results for Motzkin polynomial}
  \begin{tabular}{| c | c | c | c | c | c |c|c|}\hline
\multicolumn{2}{|c|}{ } & \multicolumn{3}{c|}{Sample size 10}& \multicolumn{3}{c|}{Sample size 100}\\ \hline
  $k$& $f_k^H$   &    Mean  & Variance  & Minimum  &  Mean      & Variance    & Minimum  \\ \hline
 %   $1$ &   $4.2000$ &   $6.2601$ &  $66.2605$  &  $0.6183$&    $6.5027$&  $188.1445$ &   $0.0060$ \\ \hline
%    $2$ &   $2.1886$ &   $1.4972$ &   $1.6084$  &  $0.9158$&    $1.8377$&   $12.5387$ &   $0.0657$ \\ \hline
%    $3$ &   $2.1886$ &   $1.9658$ &   $5.0427$  &  $0.0644$&    $2.8413$&   $68.2093$ &   $0.0036$ \\ \hline
%    $4$ &   $1.2743$ &   $1.1776$ &   $1.8501$  &  $0.0421$&    $0.8571$ &   $0.6764$ &   $0.0042$\\ \hline
    $5$ &   $1.2743$ &   $0.8330$ &   $0.0466$  &  $0.2790$&    $1.1590$ &   $4.2023$ &   $0.0525$\\ \hline
  %  $6$ &   $1.0218$ &   $1.7002$ &   $6.2647$  &  $0.3196$&    $0.9336$ &   $0.8998$ &   $0.0002$\\ \hline
%    $7$ &   $1.0218$ &   $0.8350$ &   $0.1672$  &  $0.2416$&    $0.9863$ &   $1.3777$ &   $0.0070$\\ \hline
%    $8$ &   $0.8912$ &   $0.6108$ &   $0.1451$  &  $0.0218$&    $0.8431$ &   $1.4834$ &   $0.0070$\\ \hline
%    $9$ &   $0.8912$ &   $0.7545$ &   $0.0679$  &  $0.1656$&    $0.8879$ &   $0.2752$ &   $0.0175$\\ \hline
   $10$ &   $0.8538$ &   $0.7005$ &   $0.0800$  &  $0.1862$&    $0.8435$ &   $0.1448$ &   $0.1149$\\ \hline
%   $11$ &   $0.8538$ &   $0.8244$ &   $0.0779$  &  $0.1123$&    $0.8673$ &   $0.2565$ &   $0.1100$\\ \hline
%   $12$ &   $0.8384$ &   $0.8912$ &   $0.0213$  &  $0.5919$&    $0.7835$ &   $0.2554$ &   $0.0188$\\ \hline
%   $13$ &   $0.8384$ &   $0.8286$ &   $0.0412$  &  $0.3205$&    $0.7664$ &   $0.0714$ &   $0.0112$\\ \hline
%   $14$ &   $0.8366$ &   $0.7698$ &   $0.0781$  &  $0.2083$&    $0.9574$ &   $1.2157$ &   $0.0778$\\ \hline
   $15$ &   $0.8339$ &   $0.9063$ &   $0.0153$  &  $0.6069$&    $0.8465$ &   $0.0932$ &   $0.0593$\\ \hline
 %  $16$ &   $0.8336$ &   $0.7482$ &   $0.0750$  &  $0.1759$&    $0.7209$ &   $0.0875$ &   $0.0648$\\ \hline
%   $17$ &   $0.8242$ &   $0.7430$ &   $0.0706$  &  $0.1500$&    $0.8051$ &   $0.0718$ &   $0.0984$\\ \hline
%   $18$ &   $0.8139$ &   $0.8546$ &   $0.0493$  &  $0.4460$&    $0.7749$ &   $0.0785$ &   $0.0038$\\ \hline
%   $19$ &   $0.8062$ &   $0.6621$ &   $0.0892$ &   $0.1836$&    $0.7850$ &   $0.1273$ &   $0.0408$\\ \hline
   $20$ &   $0.8025$ &   $0.7704$ &   $0.0336$ &   $0.3826$&    $0.9326$ &   $1.6454$ &   $0.0040$\\ \hline
   $25$ &   $0.7762$ &   $0.7995$ &   $0.1014$ &   $0.2433$&    $0.7493$ &   $0.0717$ &   $0.0722$\\ \hline
   $30$ &   $0.7474$ &   $1.0104$ &   $1.2852$ &   $0.1091$&    $0.8290$ &   $0.8620$ &   $0.0522$\\ \hline
   $35$ &   $0.7067$ &   $0.5930$ &   $0.0981$ &   $0.1940$&    $0.7647$ &   $1.3012$ &   $0.0016$\\ \hline
   $40$ &   $0.6625$ &   $0.6967$ &   $0.0497$ &   $0.2867$&    $0.6028$ &   $0.1371$ &   $0.0021$\\ \hline
   $45$ &   $0.6254$ &   $0.6258$ &   $0.0500$ &   $0.3548$&    $0.7007$ &   $0.2242$ &   $0.0090$\\ \hline
   $50$ &   $0.5914$ &   $0.6244$ &   $0.0718$ &   $0.3000$&    $0.5782$ &   $0.1406$ &   $0.0154$\\ \hline
\multicolumn{2}{|c|}{Uniform Sample}&$4.2888$&$37.4427$&$0.5290$& $3.7397$&$53.8833$&$0.0492$  \\ \hline
  \end{tabular}
  \end{table}
\end{center}
}

{\tiny
\begin{center}
\begin{table}[h!]
\caption{\label{tab:sampling2}Sampling results for Three-Hump Camel function}
  \begin{tabular}{| c | c | c | c | c | c |c|c|}\hline
\multicolumn{2}{|c|}{ } & \multicolumn{3}{c|}{Sample size 10}& \multicolumn{3}{c|}{Sample size 100}\\ \hline
  $k$& $f_k^H$   &    Mean  & Variance  & Minimum  &  Mean      & Variance    & Minimum  \\ \hline
%   $1$ &$265.77$& $359.98$& $274477.0$ &   $2.4493$& $300.34$& $245144.0$& $0.011095$\\ \hline
%  $2$ &$86.091$& $88.717$&  $24117.0$ &   $1.1729$& $122.12$&  $76646.0$& $0.082513$\\ \hline
%  $3$ &$86.091$& $14.712$&   $186.23$ &    $2.219$& $58.186$&  $15987.0$&    $0.492$\\ \hline
%  $4$ &$40.593$& $55.091$&  $19297.0$ &  $0.10296$& $44.844$&  $21297.0$&  $0.19439$\\ \hline
  $5$ &$40.593$& $91.872$&  $27065.0$ &  $0.90053$& $53.656$&  $14575.0$&  $0.58086$\\ \hline
 % $6$ &$24.354$& $12.961$&   $77.377$ &   $0.8186$& $34.115$&   $7862.5$& $0.019021$\\ \hline
%  $7$ &$24.354$&  $33.96$&   $1745.4$ &  $0.65266$& $27.072$&  $10632.0$&  $0.33813$\\ \hline
%  $8$ &$17.322$& $10.029$&   $60.746$ &   $1.0931$& $12.307$&   $314.46$& $0.074663$\\ \hline
%  $9$ &$17.322$& $9.4932$&   $100.22$ &$0.0027565$& $20.185$&   $7279.8$&  $0.11239$\\ \hline
 $10$ &$13.867$& $11.312$&   $45.784$ &   $0.8916$& $14.273$&   $382.98$& $0.018985$\\ \hline
% $11$ &$13.867$& $8.3991$&   $87.108$ &$0.0031527$& $11.928$&   $357.45$&  $0.01384$\\ \hline
% $12$ &$10.534$&  $5.013$&   $52.681$ &  $0.30303$& $12.377$&   $547.42$&  $0.25952$\\ \hline
% $13$ &$10.534$& $14.281$&   $401.82$ &  $0.52373$& $7.8673$&   $253.02$&  $0.11989$\\ \hline
% $14$ &$8.6752$& $5.2897$&    $43.81$ &   $0.3909$& $9.4462$&   $362.49$& $0.051331$\\ \hline
 $15$ &$8.6752$& $5.6281$&   $31.311$ &  $0.21853$& $10.373$&   $778.32$& $0.022282$\\ \hline
% $16$ &$7.2466$& $9.5801$&   $95.901$ &   $1.7112$&  $6.465$&   $122.72$& $0.013084$\\ \hline
% $17$ &$7.2466$& $5.2511$&   $23.863$ &   $2.0409$& $6.0633$&   $56.495$&  $0.18354$\\ \hline
% $18$ &$6.1763$& $6.0327$&   $34.298$ &  $0.85182$& $5.2985$&   $35.953$& $0.071544$\\ \hline
% $19$ &$6.1763$& $5.3006$&   $52.994$ &   $0.6699$& $5.0383$&   $41.619$& $0.040785$\\ \hline
 $20$ &$5.3826$& $3.5174$&  $16.053$  & $0.43269$& $9.4178$ &  $653.27$& $0.041752$\\ \hline
 $25$ &$4.2267$& $10.741$&  $776.55$  & $0.59616$& $5.0642$ &  $112.61$& $0.039463$\\ \hline
 $30$ &$3.1892$& $2.2515$&  $8.6915$  &$0.063265$& $2.2096$ &  $6.2611$& $0.040845$\\ \hline
 $35$ &$2.7367$& $1.5032$&  $1.4626$  & $0.0085016$& $3.0679$&  $16.47$&  $0.24175$\\ \hline
 $40$ &$2.2626$& $1.3941$&  $1.1995$  & $0.21653$& $2.3431$ &  $17.735$& $0.069473$\\ \hline
 $45$ &$2.0337$& $2.3904$&  $10.934$  & $0.57818$& $1.8928$ &  $3.6581$& $0.050042$\\ \hline
 $50$ &$1.7768$&  $1.664$&  $3.3983$  &$0.061995$& $1.6301$ &  $1.6966$& $0.048476$\\ \hline
\multicolumn{2}{|c|}{Uniform Sample}&$306.96$  & $275366.0$  & $0.15602$& $368.28$& $296055.0$&  $0.59281$\\ \hline
  \end{tabular}
  \end{table}
\end{center}
}

For each
 degree $k$, we use the sample sizes $10$ and  $100$.
In Tables \ref{tab:sampling1} and
\ref{tab:sampling2} we record the mean, variance and the minimum value of these samples.
(Recall that the expected value of the sample mean equals $f^H_k$.) We also generate samples uniformly from $[0,1]^n$, for comparison.

The mean of the sample function values approximates $f^H_k$ reasonably well for sample size $100$, but less so for sample
size $10$. Moreover, the mean sample function value for uniform sampling from $[0,1]^n$ is much higher than $f^H_k$.
Also, the minimum function value for sampling is significantly lower than the minimum function value
obtained by uniform sampling for most values of $k$.

\section{Concluding remarks}
\label{sec:conclusion}

One may consider several strategies to improve the upper bounds $f^H_k$, and we list some in turn.
\begin{itemize}
\item
A natural idea is to use density functions that are convex combinations of SOS and Handelman-type densities, i.e.,\ that belong
to $\HH_k + \Sigma[x]_r$ for some nonnegative integers $k,r$.
Unfortunately one may show that this does not yield a better upper bound than $\min \{f^{sos}_r, f^H_k\}$, namely
\[
\min \{f^{sos}_r, f^H_k\} = \,\inf_{\sigma\in\HH_k+\Sigma[x]_r\,}\,\left\{\displaystyle\int_\K f(\x)\,\sigma(\x)\,d\x:\:
\displaystyle\int_\K \sigma(\x)\,d\x=1\right\},\quad k,r\in \N.
\]
(We omit the proof since it is straightforward, and of limited interest.)
\item
For optimization over the hypercube,  a second idea is to replace the integer exponents in Handelman representations of the density by more general positive real exponents.
(This is amenable to analysis since the beta distribution is defined for arbitrary positive shape parameters \tcolred{ and with its moments  available via relation (\ref{kmoment})}.)
 If we drop the integrality requirement for $(\eta,\beta)$ in the definition of $f^H_k$ (see \eqref{maxcutformula}), we obtain the bound:
\begin{equation*}
%\label{eq:betabound}
f^H_k \ge {f_k^{beta}}\,:=\,\displaystyle\min_{(\eta,\beta)\in\Delta^{2n}_k}\:\sum_{\alpha\in\N^{n}_{\le d}}f_{\alpha}\,\frac{\gamma_{(\eta+\alpha,\beta)}}{\gamma_{(\eta,\beta)}}, \quad k \in \N,
\end{equation*}
where $\Delta^{2n}_k$ is the simplex $\Delta^{2n}_k := \{ (\eta,\beta) \in \mathbb{R}^{2n}_+ \; : \; \sum_{i=1}^n (\eta_i+\beta_i) = k\}$.

As with $f^H_k$, when $(\eta,\beta)$ is such that $f_k^{beta}=\sum_{\alpha\in\N^{n}_{\le d}}f_{\alpha}\,\frac{\gamma_{(\eta+\alpha,\beta)}}{\gamma_{(\eta,\beta)}},$
one has that $ {f_k^{beta}} = \mathbb{E}(f(X))$ where $X = (X_1,\ldots, X_n)$ and $X_i \sim beta(\eta_i+1,\beta_i+1)$ ($i \in [n]$). Using the
moments of the beta distribution in \eqref{kmoment}, we obtain
\begin{equation}\label{eq:betabound}
f_k^{beta} = \min_{(\eta,\beta)\in\Delta^{2n}_k}\:\sum_{\alpha\in\N^{n}_d}f_{\alpha} \prod_{i=1}^n {(\eta_i+1)\cdots(\eta_i+\alpha_i) \over (\eta_i+\beta_i+2)\cdots(\eta_i+\beta_i+\alpha_i+1)} , \quad k \in \N.
\end{equation}
Thus one may obtain the bounds $f_k^{beta}$ by minimizing a rational function over a simplex.
A question for future research is whether one may approximate $f^{beta}_k$ to any fixed accuracy  in time polynomial in $k$ and $n$.
(This may be possible, since the minimization of fixed-degree polynomials over a simplex allows a PTAS \cite{KLP06}, and the relevant algorithmic techniques have been extended to
rational objective functions \cite{Jibetean}.)

One may also use the value of $(\eta,\beta)\in\Delta^{2n}_k$ that gives $f^H_k$ as a starting point in the minimization problem \eqref{eq:betabound}, and employ
any iterative method to obtain a better upper bound heuristically.
Subsequently, one may use the resulting density function to obtain `good' feasible points as described in Section \ref{sec:feasible points}.
Of course, one may also use the feasible points (generated by sampling) as starting points for iterative methods. Suitable iterative
methods for bound-constrained optimization are described in the books \cite{Bertsekas_book,Fletcher_book,Gill Murray Wright}, and the latest algorithmic developments
for bound constrained global optimization are surveyed in the recent thesis \cite{Pal thesis}.

\item
Perhaps the most promising practical variant of the $f^H_k$ bound is the following parameter:
\begin{eqnarray*}
f_{r,k}^H &=& \displaystyle\min_{(\eta,\beta)\in\N^{2n}_k}\:\frac{\displaystyle\int_\K f(\x)\,\left(\x^\eta(\1-\x)^\beta\right)^r\,d\x}
{\displaystyle\int_\K (\x^\eta(\1-\x)^\beta)^r\,d\x} \\
 &=&\min_{(\eta,\beta)\in\N^{2n}_k}\:\sum_{\alpha\in\N^{n}}f_{\alpha}\,\frac{\gamma_{(r\eta+\alpha,r\beta)}}{\gamma_{(r\eta,r\beta)}}
\quad \quad \text{ for } r,k \in \mathbb{N}.
\end{eqnarray*}
 Thus, the idea is to replace the density $\sigma(x) = \x^\eta(\1-\x)^\beta/\int_\K \x^\eta(\1-\x)^\beta\,d\x$ by the density $\sigma(x)^r/\int_\K \sigma(r)^r\,d\x$ for some power $r \in \mathbb{N}$. Hence, for $r=1$, $f^H_{1,k}=f^H_k$.
 Note that the calculation of $f_{r,k}^H$ requires exactly the same number of elementary operations as the calculation of $f^H_k$, provided all the required moments are available.
 \tcolred{(Also note that, for $\K=[0,1]^n$, one could allow  an arbitrary $r>0$  since the moments are still available as pointed out  above.)}
 %(Also note that $f^H_k = f^H_{1,k}$.)

In Tables \ref{tab:fkr1}, \ref{tab:fkr2}, and \ref{tab:fkr3}, we show some relative gaps
for the parameter $f_{r,k}^H$, defined as $(f_{r,k}^H-f_{\min,\K})/(f_{\max,\K}-f_{\min,\K}) \times 100$.

{\tiny
\begin{center}
\begin{table}[h!]
\caption{\label{tab:fkr1} Relative gaps of $f_{r,k}^H$ for the Styblinski-Tang function ($n=2$)}
  \begin{tabular}{|c| c | c | c | c | c |}\hline
$k$&$r=1$&$r=2$&$r=3$&$r=4$&$r=5$\\ \hline
$1$&$20.0499$&$20.7931$&$21.3190$&$21.3190$&$21.3190$\\ \hline
$2$&$18.5633$&$18.4184$&$18.7040$&$19.0470$&$19.3665$\\ \hline
$3$&$17.2942$&$17.2522$&$16.9793$&$16.7974$&$16.6631$\\ \hline
$4$&$15.8076$&$15.5176$&$15.2511$&$14.6398$&$14.1912$\\ \hline
$5$&$15.0461$&$14.3517$&$14.3645$&$13.8452$&$13.3692$\\ \hline
$6$&$14.2847$&$13.1855$&$12.6361$&$12.2758$&$12.0074$\\ \hline
$7$&$13.8738$&$12.0519$&$10.9113$&$10.1182$&$ 9.5355$\\ \hline
$8$&$13.4630$&$10.9180$&$ 9.1831$&$ 7.9606$&$ 7.0636$\\ \hline
$9$&$13.2211$&$10.3381$&$ 8.4528$&$ 7.1660$&$ 6.2416$\\ \hline
$10$&$12.9796$&$ 9.7582$&$ 7.7221$&$ 6.3713$&$ 5.4195$\\ \hline
\end{tabular}
\end{table}
\end{center}

\begin{center}
\begin{table}[h!]
\caption{\label{tab:fkr2} Relative gaps of $f_{r,k}^H$ for the Rosenbrock function ($n=3$)}
  \begin{tabular}{|c| c | c | c | c | c | }\hline
  $k$&$r=1$&$r=2$&$r=3$&$r=4$&$r=5$\\ \hline
$1$&$10.1745$&$ 9.3107$&$ 8.9356$&$ 8.7536$&$ 8.6603$\\ \hline
$2$&$ 7.7310$&$ 6.5571$&$ 6.0674$&$ 5.8142$&$ 5.6807$\\ \hline
$3$&$ 6.8671$&$ 5.7557$&$ 5.1021$&$ 4.7091$&$ 4.4890$\\ \hline
$4$&$ 6.1275$&$ 4.7220$&$ 3.7699$&$ 3.2404$&$ 2.9126$\\ \hline
$5$&$ 5.2637$&$ 3.5090$&$ 3.0196$&$ 2.9302$&$ 2.9826$\\ \hline
$6$&$ 4.4018$&$ 2.8821$&$ 2.4570$&$ 1.9388$&$ 1.5359$\\ \hline
$7$&$ 4.0267$&$ 2.8901$&$ 2.1273$&$ 1.6465$&$ 1.3623$\\ \hline
$8$&$ 3.7922$&$ 2.5456$&$ 1.8554$&$ 1.4301$&$ 1.1273$\\ \hline
$9$&$ 3.4171$&$ 2.3701$&$ 1.7074$&$ 1.3206$&$ 1.0798$\\ \hline
$10$&$ 3.2259$&$ 2.0283$&$ 1.4251$&$ 1.1250$&$ 0.8966$\\ \hline
\end{tabular}
\end{table}
\end{center}

\begin{center}
\begin{table}[h!]
\caption{\label{tab:fkr3} Relative gaps of $f_{r,k}^H$ for the Rosenbrock function ($n=4$)}
  \begin{tabular}{|c| c | c | c | c | c | }\hline
  $k$&$r=1$&$r=2$&$r=3$&$r=4$&$r=5$\\ \hline
$1$&$11.0081$&$10.4440$&$10.1939$&$10.0727$&$10.0104$\\ \hline
$2$&$ 9.3678$&$ 8.5929$&$ 8.2655$&$ 8.0963$&$ 8.0074$\\ \hline
$3$&$ 7.7383$&$ 6.7421$&$ 6.3371$&$ 6.1202$&$ 6.0046$\\ \hline
$4$&$ 7.1624$&$ 6.2079$&$ 5.7098$&$ 5.4000$&$ 5.2266$\\ \hline
$5$&$ 6.6694$&$ 5.1729$&$ 4.2870$&$ 3.8120$&$ 3.5307$\\ \hline
$6$&$ 6.0935$&$ 4.4015$&$ 3.3909$&$ 2.8242$&$ 2.4706$\\ \hline
$7$&$ 5.5188$&$ 3.5929$&$ 2.8908$&$ 2.6175$&$ 2.5173$\\ \hline
$8$&$ 4.9429$&$ 3.1671$&$ 2.5076$&$ 1.9564$&$ 1.5528$\\ \hline
$9$&$ 4.3682$&$ 2.8285$&$ 2.2958$&$ 1.7616$&$ 1.4370$\\ \hline
$10$&$ 4.1182$&$ 2.7624$&$ 2.1065$&$ 1.6160$&$ 1.2793$\\ \hline
\end{tabular}
\end{table}
\end{center}
}

\ignore{
{\tiny
\begin{center}
\begin{table}[h!]
\caption{\label{tab:fkr1} $f_{r,k}^H$ for the Styblinski-Tang function ($n=2$)}
  \begin{tabular}{|c| c | c | c | c | c |}\hline
$k$&$r=1$&$r=2$&$r=3$&$r=4$&$r=5$\\ \hline
$1$ &$-12.5$&$-10.06$&$-8.3333$&$-8.3333$&$-8.3333$\\ \hline
$2$ &$-17.381$&$-17.857$&$-16.919$&$-15.793$&$-14.744$\\ \hline
$3$ &$-21.548$&$-21.686$&$-22.582$&$-23.179$&$-23.62$\\ \hline
$4$ &$-26.429$&$-27.381$&$-28.256$&$-30.263$&$-31.736$\\ \hline
$5$ &$-28.929$&$-31.209$&$-31.167$&$-32.872$&$-34.435$\\ \hline
$6$ &$-31.429$&$-35.038$&$-36.842$&$-38.025$&$-38.906$\\ \hline
$7$ &$-32.778$&$-38.76$&$-42.505$&$-45.109$&$-47.022$\\ \hline
$8$ &$-34.127$&$-42.483$&$-48.179$&$-52.193$&$-55.138$\\ \hline
$9$ &$-34.921$&$-44.387$&$-50.577$&$-54.802$&$-57.837$\\ \hline
$10$ &$-35.714$&$-46.291$&$-52.976$&$-57.411$&$-60.536$\\ \hline
\end{tabular}
\end{table}
\end{center}

\begin{center}
\begin{table}[h!]
\caption{\label{tab:fkr2} $f_{r,k}^H$ for the Rosenbrock function ($n=3$)}
  \begin{tabular}{|c| c | c | c | c | c | }\hline
  $k$&$r=1$&$r=2$&$r=3$&$r=4$&$r=5$\\ \hline
$1$& $794.818$&$727.337$&$698.032$&$683.822$&$676.526$\\ \hline
$2$& $603.931$&$512.228$&$473.974$&$454.193$&$443.769$\\ \hline
$3$& $536.449$&$449.625$&$398.566$&$367.869$&$350.671$\\ \hline
$4$& $478.673$&$368.873$&$294.499$&$253.135$&$227.526$\\ \hline
$5$& $411.191$&$274.121$&$235.89$&$228.906$&$232.996$\\ \hline
$6$& $343.863$&$225.146$&$191.935$&$151.455$&$119.98$\\ \hline
$7$& $314.559$&$225.768$&$166.179$&$128.62$&$106.417$\\ \hline
$8$& $296.24$& $198.861$&$144.94$&$111.721$&$88.0661$\\ \hline
$9$& $266.936$&$185.145$&$133.379$&$103.162$&$84.3506$\\ \hline
$10$&$252.003$&$158.448$&$111.33$&$87.8805$&$70.0394$\\ \hline
\end{tabular}
\end{table}
\end{center}

\begin{center}
\begin{table}[h!]
\caption{\label{tab:fkr3} $f_{r,k}^H$ for the Rosenbrock function ($n=4$)}
  \begin{tabular}{|c| c | c | c | c | c | }\hline
  $k$&$r=1$&$r=2$&$r=3$&$r=4$&$r=5$\\ \hline
$1$ &$1289.9$&$1223.8$&$1194.5$&$1180.3$&$1173.0$\\ \hline
$2$ &$1097.7$&$1006.9$&$968.53$&$948.71$&$938.29$\\ \hline
$3$ &$906.76$&$790.03$&$742.57$&$717.15$&$703.61$\\ \hline
$4$ &$839.28$&$727.43$&$669.06$&$632.76$&$612.44$\\ \hline
$5$ &$781.51$&$606.15$&$502.34$&$446.68$&$413.72$\\ \hline
$6$ &$714.02$&$515.76$&$397.34$&$330.93$&$289.5$\\ \hline
$7$ &$646.68$&$421.01$&$338.74$&$306.71$&$294.97$\\ \hline
$8$ &$579.2$& $371.11$&$293.83$&$229.25$&$181.95$\\ \hline
$9$ &$511.86$&$331.44$&$269.02$&$206.42$&$168.39$\\ \hline
$10$&$482.56$&$323.69$&$246.84$&$189.36$&$149.9$\\ \hline
\end{tabular}
\end{table}
\end{center}
}}

A first important observation is that, for fixed $k$, the values of $f^H_{r,k}$ are not monotonically decreasing in $r$; see e.g.\ the row $k=2$ in Table \ref{tab:fkr1}. Likewise, the sequence $f^H_{r,k}$ is not monotonically decreasing in $k$ for fixed $r$; see, e.g., the column $r=5$ in Table \ref{tab:fkr2}.

On the other hand, it is clear from Tables \ref{tab:fkr1}, \ref{tab:fkr2}, and \ref{tab:fkr3}
that $f^H_{r,k}$ can provide a much better bound than $f^H_{k}$ for $r > 1$.

Since $f^H_{r,k}$ is not monotonically decreasing in $r$ (for fixed $k$), or in $k$ (for fixed $r$), one has to consider the convergence question.
An easy case is when  $\K = [0,1]^n$ and the global minimizer $\x^*$ is rational. Say  $x^*_i = \frac{p_i}{q_i}$ ($i \in [n]$),
%By the proof of Theorem \ref{cor:convergence rate} (see (\ref{eq:frk rate})), and the construction in Definition \ref{defbetapara},
\tcolred{setting $q_i=1$ and $p_i=x^*_i$ when  $x^*_i\in \{0,1\}$. Consider the following variation of the parameters $\eta^*_i,\beta^*_i$ from Definition \ref{defbetaparaQ}: $\eta^*_i=rp_i+1$ and $\beta^*_i=r(q_i-p_i)+1$ for $i\in [n]$, so that $\sum_{i=1}^n\eta^*_i+\beta^*_i-2= r(\sum_{i=1}^nq_i)$.
Combining relation (\ref{eqbeta}) and Theorem \ref{thmdif}, we can conclude that
   the following inequality holds: }
\[
f^H_{r,k} -f(\x^*) \le {C_f\over r} \quad \mbox{ for all \   $k \ge \sum_{i=1}^n q_i$ \ and \  $r\ge 1$},
\]
where $C_f$ is a constant that depends on $f$ only.

For more general sets $\K$, one may
ensure convergence by considering instead the following parameter (for fixed $R \in \mathbb{N}$):
\[
\min_{r \in [R]} f^H_{k,r} \le f^H_k \quad (k \in \mathbb{N}).
\]
Then convergence follows from the convergence results for $f^H_{k,r}$.
Moreover, this last parameter may be computed in polynomial time if $k$ is fixed, and $R$ is bounded by a  polynomial in $n$.
\end{itemize}

\subsection*{Acknowledgements}
Etienne de Klerk would like to thank Dorota Kurowicka for valuable discussions on the beta distribution.
The research of Jean B. Lasserre was funded by by the European Research Council (ERC) under the European Union's Horizon 2020 research and innovation program (grant agreement 666981 TAMING).
{We thank two anonymous referees for their useful suggestions that helped improve the presentation of the paper.}

\end{document}